\documentclass[12pt]{article}
\usepackage{graphicx} %

\usepackage{amssymb}
\usepackage{amsmath,amsthm}
\usepackage[english]{babel}
\usepackage[T1]{fontenc}
\usepackage[a4paper,top=2.5cm,bottom=3cm,left=2.5cm,right=2.5cm,marginparwidth=1.75cm]{geometry}

\usepackage{microtype}

\usepackage{enumitem} %

\usepackage{mathtools}
\usepackage{graphicx,adjustbox}
\usepackage[dvipsnames]{xcolor}
\usepackage{bbm}
\usepackage[colorlinks=true, allcolors=black]{hyperref}
\usepackage[normalem]{ulem}
\usepackage{cancel}
\usepackage{varwidth}  %
\usepackage{booktabs}

\usepackage{xparse}
\NewDocumentCommand{\grad}{e{_^}}{%
  \mathop{}\!%
  \nabla
  \IfValueT{#1}{_{\!#1}}%
  \IfValueT{#2}{^{#2}}%
}

\DeclareMathOperator{\tr}{tr}

\DeclareMathOperator{\supp}{supp}

\usepackage{xspace}

\usepackage{dsfont}

\newtheorem{theorem}{Theorem}[section]
\newtheorem{corollary}[theorem]{Corollary}
\newtheorem{lemma}[theorem]{Lemma}
\newtheorem{proposition}[theorem]{Proposition}

\theoremstyle{definition}
\newtheorem{definition}[theorem]{Definition}
\usepackage{textcomp}
\newtheorem{remarkx}[theorem]{Remark}
\newtheorem{examplex}[theorem]{Example}
\newenvironment{remark}
{\pushQED{\qed}\remarkx} %
{\popQED\endremarkx}
{\popQED\endexamplex}

\DeclareMathAlphabet{\mathup}{OT1}{\familydefault}{m}{n}
\newcommand{\dx}[1]{\mathop{}\!\mathup{d} #1}
\newcommand{\dd}{\dx} %
\usepackage{ifthen}
\newlength{\leftstackrelawd}
\newlength{\leftstackrelbwd}
\def\leftstackrel#1#2{\settowidth{\leftstackrelawd}%
{${{}^{#1}}$}\settowidth{\leftstackrelbwd}{$#2$}%
\addtolength{\leftstackrelawd}{-\leftstackrelbwd}%
\leavevmode\ifthenelse{\lengthtest{\leftstackrelawd>0pt}}%
{\kern-.5\leftstackrelawd}{}\mathrel{\mathop{#2}\limits^{#1}}}

 \def\calB{{\mathcal B}} 
  \def\calF{{\mathcal F}}

 \def\calN{{\mathcal N}} 
\def\calP{{\mathcal P}}  
  
  \def\calX{{\mathcal X}}

\def\rmd{{\mathrm d}}

\def\rmD{{\mathrm D}}

 \def\bbE{{\mathbb E}} \def\bbF{{\mathbb F}}

\def\bbP{{\mathbb P}}  \def\bbR{{\mathbb R}}
\def\bbS{{\mathbb S}}  
  
 \def\bbZ{{\mathbb Z}}

\DeclareMathOperator\Id{Id}

\DeclareMathOperator{\Sym}{Sym}
\DeclareMathOperator{\law}{law}
\DeclareMathOperator{\dist}{dist}
\DeclareMathOperator{\vol}{vol}
\DeclareMathOperator{\Crit}{Crit}

\newcommand{\faa}          {\quad \text{for almost all } \,}
\title{Random Quadratic Form on a Sphere: \\ Synchronization by Common Noise}
\author{Maximilian Engel$^{1, 2}$, Anna Shalova$^{1}$\thanks{\href{mailto:a.shalova@uva.nl}{a.shalova@uva.nl}}
\\
\normalsize $^{1}$ Korteweg-de Vries Institute for Mathematics, University of Amsterdam,
\\
\normalsize $^{2}$ Department of Mathematics and Computer Science, FU Berlin
}
\date{\today}

\begin{document}

\maketitle

\begin{abstract}

    We introduce the Random Quadratic Form (RQF): a stochastic differential equation which formally corresponds to the gradient flow of a random quadratic functional on a sphere. While the one-point dynamics of the system is a Brownian motion and thus has no preferred direction, the two-point motion exhibits nontrivial synchronizing behaviour. In this work we study synchronization of the RQF, namely we give both distributional and path-wise characterizations of the solutions by studying invariant measures and random attractors of the system. 
    
    The RQF model is motivated by the study of the role of linear layers in transformers and illustrates the \emph{synchronization by common noise} phenomena arising in the simplified models of transformers. In particular, we provide an alternative (independent of self-attention) explanation of the clustering behaviour in deep transformers and show that tokens cluster even in the absence of the self-attention mechanism.
\end{abstract}
\tableofcontents
\section{Introduction}
In this work we study the Random Quadratic Form (RQF), the stochastic differential equation (SDE) on a sphere $\bbS^{n-1}$ with  multiplicative noise defined by the process $Q_t$, namely
    \begin{equation}
    \label{eq:rqf}
    \rmd X_t  %
    = -P_{X_t}  \partial Q_t  X_t,
    \end{equation}
    where for every $X \in \bbS^{n-1}$
    \[
    P_{X} := P_{T_{X}\bbS^{n-1}} = I - XX^T
    \]
    is the projection onto the tangent space of $\bbS^{n-1}$ at $X$. The noisy process 
    $$Q_t: (0, T) \times \Omega \to \Sym^n$$
    is a stochastic process on the space of symmetric real $n\times n$ matrices, given as
\begin{equation}
\label{eq:A}
Q_t =  \frac{1}{\sqrt{2}}(B_t + B_t^T), 
\end{equation}
where $\{B_t^{ij}: i, j \in 1\dots n\}$ are independent Brownian motions.  Finally, the notation $\partial Q_t$ implies that the SDE \eqref{eq:rqf} is understood in the Stratonovich sense. We use the $\partial Q_t$ notation for the Stratonovich integrals due to the matrix form of the noise process~$Q_t$. Our analysis requires both left and right matrix-vector multiplications and thus we write the integrator $\partial Q_t$ in the middle; this argument becomes more transparent in Section \ref{sec:intro-gf}. 

 The RQF \eqref{eq:rqf} can be interpreted as a gradient flow of a \emph{random quadratic functional} on a sphere; and the gradient structure of the dynamics provides important insights on the long-time behaviour of the system. %
 In particular, the standard (deterministic) gradient flow structure implies that the driving functional is the \emph{optimal} 
Lyapunov function of the underlying dynamics (see Remark~\ref{remark:Lyap}), guaranteeing, under certain convexity assumptions, the convergence of the solutions to the minimizer of the driving functional. Remarkably, some of the properties of the deterministic gradient dynamics translate into the random case, which we illustrate through the example of RQF (and also the spherical Brownian motion, see Section~\ref{sec:bias}). In particular, despite the minimizers of the RQF being \emph{random} sets, the underlying dynamics behaves similarly to the gradient flow of a deterministic quadratic functional clustering the points into an anti-polar configuration. 

The summary of the main results, describing the long-time behaviour of the RQF, is given in Section \ref{sec:intro-results}. Before we proceed to the results, we give a heuristic argument revealing the gradient structure of Eq.~\eqref{eq:rqf} in Section \ref{sec:intro-gf} to strengthen the relation between deterministic and random quadratic forms. We also introduce the continuous-time counterpart of the transformer model as a key motivation for this study in Section \ref{sec:intro-transformers} and give a literature overview in Section \ref{sec:lit}.

\subsection{Gradient flow formulation}
\label{sec:intro-gf}
In this section we give a \emph{formal} calculation, illustrating the gradient structure of the RQF. In particular, one of the steps involves differentiating a differential, which we do not intend to make rigorous. We provide it here for the illustrative purposes only.

Let $M \in \Sym^n$ be a symmetric real matrix and let $F_M: \bbS^{n-1} \to \bbR$ be the corresponding quadratic functional, namely
\[
F_M (x) := \frac12x^TMx.
\]
We call the gradient flow of $F_M$ on $\bbS^{n-1}$, given by
\begin{equation}
\label{eq:intro-dqf}
    \dot x := - \nabla F_M(x) = -P_x Mx, \qquad x(0) = x_0 \in \bbS^{n-1},
\end{equation}
where $\nabla$ denotes the Riemannian gradient on $\bbS^{n-1}$ with the natural topology of the flow, the \emph{deterministic quadratic form}. Rewriting \eqref{eq:intro-dqf} in integral form we thus obtain
\[
x(t) =  x_0 - \int_0^t \nabla F_M(x(s)) \rmd s =x_0 - \int_0^t P_{x}Mx(s) \rmd s. 
\]
Moreover, introducing the cumulative functional $\bbF_M$, averaging $F_M$ along the solutions of the gradient flow, we can rewrite \eqref{eq:intro-dqf} as the following coupled system of equations
\begin{align}
x(t) &=  x_0  - \int_{0}^t \nabla \partial_s \bbF_M(s, x_0), \quad \text{where} \label{eq:intro-curve} \\
\bbF_M(t, x_0) &:= \int_0^t  F_M(x(s)) \rmd s = \frac12\int_0^t x(s)^TMx(s) \rmd s,
\end{align}
and $\partial_s \bbF_M(s, x_0)$ is a formal differential of $\bbF_M$, namely
\[
\partial_s \bbF_M(s, x_0) := \frac{\rmd}{\rmd t} \bbF_M(t, x_0) \big|_{t=s} \cdot \rmd s= F_M(x(s))\rmd s.
\]
In other words, the object $\partial_s \bbF_M(s)$ is understood in the sense that the curve~\eqref{eq:intro-curve} solves~\eqref{eq:intro-dqf}.

Analogously, let $M : (0, T) \to \Sym^n$ be a curve on $\Sym^n$ and consider the following quadratic \emph{driving functional} defined by $M$:
\[
F_M (x, t) := \frac12x^TM(t)x.
\]
Then, the gradient flow dynamics of the the time-dependent functional $F_M$ given by
\begin{equation}
\label{eq:intro-gf}
\dot x = -\nabla F_M(x, t) = -P_{x}M(t)x, \qquad x(0) = x_0 \in \bbS^{n-1}, 
\end{equation}
is equivalent to the system
\begin{align*}
x(t) &=  x_0  - \int_{0}^t \nabla \partial_s \bbF_M(s, x_0), \quad \text{where} \\
\bbF_M(t, x_0) &:= \int_0^t  F_M(x(s), s) \rmd s = \frac12\int_0^t x(s)^TM(s)x(s) \rmd s, \\
\partial_s \bbF_M(s, x_0) &:= \frac{\rmd}{\rmd t} \bbF_M(t, x_0) \big|_{t=s} \cdot \rmd s= F_M(x(s), s)\rmd s.
\end{align*}
Moreover, let us define $Q_s$ to be the average of the matrix $M(s)$ over time, namely let $Q_s$ solve
\begin{equation}
\label{eq:intro-Adet}
\rmd Q_s = M(s) \rmd s.
\end{equation}
Then, the differential $\partial_s \bbF_M$ above is equal to $\partial_s \bbF_M(s, x_0) = \frac12 x(s)^T (\rmd Q_s) x(s)$ and we can formally rewrite \eqref{eq:intro-gf} once again as
\begin{align*}
x(t) &=  x_0  - \int_{0}^t \nabla  F_{\rmd Q_s}(x(s)), \quad \text{where} \\
F_{\rmd Q_s}(x(s)) &:= \frac12 x(s)^T (\rmd Q_s) x(s) = \frac12 x(s)^T M(s) x(s) \rmd s. 
\end{align*}
Formally, replacing the deterministic curve \eqref{eq:intro-Adet} with the random process $Q_t$ as defined in~\eqref{eq:A}, we  obtain the \emph{gradient flow formulation} of the RQF \eqref{eq:rqf}:
\[
X_t =  X_0  - \int_{0}^t \nabla F_{\partial Q_s}(X_s) \quad \text{where} \quad F_{\partial Q_s}(X_s):= \frac12 X_s^T\partial Q_s X_s,
\]
where the integral is understood in the Stratonovich sense. Intuitively this characterization implies that the increment of the solution of the RQF at time $s$ tends to minimize the quadratic functional $\frac12 X_s^T\partial Q_s X_s$ defined by the increments of the noise process~$Q_t$. 
\begin{remark}[Back from the random quadratic form to the deterministic one] Note that the RQF \eqref{eq:rqf} can be defined for a more general choice of the noise process $Q_t$. In particular, choosing $Q_t = tM$ for some $M \in \Sym^n$ gives the standard gradient flow of the fixed quadratic functional $F_M$:
\[
\dot x \cdot \rmd s = -P_x M x \cdot \rmd s= - \nabla \left(\frac{1}{2} x^T M x\right) \cdot \rmd s.
\]
In this work we show the similarities between the two opposite cases: deterministic and white-noise driven processes $Q_t$. We expect that the behaviour of the RQF stays similar for a larger class of driving processes, but leave it as a topic for future research.
\end{remark}
\begin{remark}[Interpretation of the driving functional]
\label{remark:Lyap}
    The driving functional of a gradient flow is a Lyapunov function of the corresponding system; moreover, it is an optimal Lyapunov function, where the optimality is understood in the following sense. Let $x_t$ be a solution of an abstract gradient flow equation $\dot x = - \nabla F(x)$, then by the chain rule we obtain
$$\frac{\rmd}{\rmd t}F(x_t) = \left< \nabla F(x_t), \dot x_t\right> = - \|\nabla F(x_t)\|\|\dot x_t\|.$$ Analogously, for any other Lyapunov function $L$ of the dynamics we calculate
$$\frac{\rmd}{\rmd t}L(x_t) = \left< \nabla L(x_t), \dot x_t\right> \geq - \|\nabla L(x_t)\|\|\dot x_t\|,$$
where the equality is achieved if and only if $\nabla L = \xi \nabla F$ almost everywhere along the trajectory $x_t$ for some non-negative $\xi = \xi(x)$.
\end{remark}
\subsection{Motivation: role of linear layers in Transformers}
\label{sec:intro-transformers}
Transformers form a class of neural network models originally introduced in \cite{vaswani2017attention} for natural language modeling problems. To build a computer-friendly representation of a text, a common approach is to construct a vocabulary consisting of all possible words (or other small lexical elements called \emph{tokens}) and assign a (unit real) vector value to every element of the vocabulary. Having built such a vocabulary, every text can then be split into a sequence of words and be represented as a sequence of vectors corresponding to the tokens in the text. In particular, a sentence of length $d$ has a representation  $(x_i)_{1\leq i\leq d}$, $x_i \in \bbR^{n}$, where $n$ is the dimension of the model. 

 The transformer architecture takes the above mentioned token representation as an input and iteratively updates the state of every token depending on the whole set of states at the previous iteration
 \[
 x^{k+1}_i = T_k ((x^k_i)_{1\leq i\leq d}), 
 \]
 where $T^k$ is the $k$-th transformer layer. Every update $T^k$ consists of the so-called feed-forward layer, a self-attention layer and the normalization (projection to the unit sphere). We refer the reader to \cite{geshkovski2024mathematical, baptista2026large} for an extensive description of the transformer architecture. Below we give a continuous-time counterpart of the dynamics following \cite{sander2022sinkformers, geshkovski2024mathematical} and relate the RQF model to the linear part of the feed-forward layer. 

Let $x_i(t) \in \bbS^{n-1}$ be the representation of the $i$-th token at time $t$. Then the continuous-time transformer dynamics is given by
\[
\begin{aligned}
\dot x_i &= P_{x_i}\left(\text{FF}(x_i) + \text{Attn}(x_i; x_1, x_2 \dots x_d)\right), \\
\text{FF}(x_i) &= \sigma(Mx_i + B), \\
\text{Attn}(x_i; x_1, x_2 \dots x_d) &= \frac{1}{\sum_j e^{x_iQ^TKx_j}}\sum_j e^{x_iQ^TKx_j}Vx_j,
\end{aligned}
\]
where $\sigma$ is an activation function, $M, B$ are the parameters of the Feed-Forward layer (denoted as $\text{FF}$) and $Q, K, V$ are the parameters of the Self-Attention layer ($\text{Attn}$), which may depend on time $t$ but not on the indices $i, j$. 
Such a structure can be interpreted as the dynamics of an interacting particle system, where the Feed-Forward layer plays the role of potential energy and the Self-Attention layer describes the interaction forces between tokens (particles), as mentioned in \cite{alvarez2026perceptrons, zimin2026yuriiformer}. 

The role of the Self-Attention layers in this (continuous-time) framework has been recently extensively studied, see \cite{rigollet2025mean} and references therein. In particular, the self-attention mechanism has been shown to exhibit clustering behaviour in various settings, see Section \ref{sec:lit}. At the same time, the Feed-Forward layers have often been discarded to simplify the analysis, with an exception of the work \cite{alvarez2026perceptrons}, where the transformer architecture consisting of a fixed Self-Attention layer and a \emph{fixed} Feed-Forward layer is considered.

In this work, we focus solely on the Feed-Forward layer and consider the following simplified model of transformers in the absence of the interaction force (self-attention):
\begin{equation}
\label{eq:neural-odes}
\dot x_i = P_{x_i}\text{FF}(x_i), \qquad \text{FF}(x_i) = \sigma(M(t)x_i + B).
\end{equation}
Under the additional structural assumptions $B \equiv 0$ and $\sigma(x) = x$, we conclude that every token follows the dynamics \eqref{eq:intro-gf}. Finally, to justify the white-noise structure of the driving process $Q_s$ we remark that the parameters in transformers are initialized randomly and are independent from layer to layer. Considering in addition the process $Q_t$ to be self-adjoint, we obtain the RQF model
\[
\rmd X_t = P_{X_t}  \partial Q_t  X_t,
\]
as a simplified model of the dynamics driven by Feed-Forward layers in transformers. Note that the processes $Q_t$ and $-Q_t$ have the same distribution and, thus, the minus sign in the definition~\eqref{eq:rqf} is largely a matter of taste and does not affect the properties of the system.

We remark that model \eqref{eq:neural-odes} falls into a more general class of neural networks, namely Neural ODEs, introduced in \cite{chen2018neural}. However, the standard problem setting of Neural ODEs is concerned with the one-point motion in a finite-time horizon and, hence, we argue that the setting of this work is closer related to transformer models: Firstly, we consider the joint dynamics of multiple particles (tokens). Secondly, we consider the long-time behavior, which is common in the transformers setting.
\subsection{Related work}
\label{sec:lit}
\paragraph{Synchronization by noise}
We begin by giving a non-exhaustive description of the results concerning synchronization in various stochastic systems appearing due to the common noise. We refer the reader to \cite{FGS1} for more references on the topic. 

Synchronization by noise in linear models was proved by means of the multiplicative ergodic theorem in \cite{arnold1983stabilization}; the result was later extended to the infinite-dimensional case in \cite{caraballo2004stabilisation}. The setting of our work was studied in \cite{baxendale1986asymptotic, baxendale1991statistical}, where the synchronization phenomenon was, similarly to the linear model, related to the Lyapunov exponents of the underlying system. The same approach has been applied to the  isotropic Brownian flows on the sphere in \cite{raimond1999flots}. An alternative approach to characterizing synchronization, based on the Feller explosion test, was introduced in \cite{scheutzow2002comparison} and further developed in \cite{cranston2016weak}. 

Additional probabilistic tests for stochastic differential equations have been introduced in \cite{FGS1} and, for order-preserving systems, in \cite{FGS2}, extending results from \cite{CrauelFlandoli98} for SDEs and \cite{Caraballoetal07} for SPDEs. The distinction between uniform and non-uniform synchronization and its relation to random bifurcations has been observed in \cite{Callawayetal17} and then been extended to multidimensional SDEs \cite{Doanetal18} and SPDEs \cite{BlumenthalEngelNeamtu2023}, with recent quantification via large deviation theory in \cite{BlessingBlumenthalBredenEngel2025}.
We remark that all of the mentioned works besides \cite{baxendale1991statistical} are focused on the case of diffusion processes synchronizing to a single point, which does not hold in the case of the RQF. Another example of \emph{partial synchronization} by common noise has been established in the context of jump processes and countable state spaces \cite{Engeletal25}, with a general classification being provided in \cite{ChemnitzEngelOlicon25}.

An approach based on large-deviations theory, allowing to prove local synchronization, was developed in \cite{mahony1996gradient} for diffusions and \cite{tearne2008collapse} for random ODEs.
For additional studies on \emph{noise-induced order} see, e.g., \cite{BlumenthalNisoli} and \cite{Endoetal25}.

\paragraph{Non-autonomous gradient systems} Relative convergence of trajectories is the distinctive feature of gradient flows of convex functionals. While it is commonly assumed that the driving functional does not explicitly depend on time, the gradient flow framework has also been extended to the non-autonomous setting, see e.g. \cite{mielke2005evolution, mainik2005existence, mielke2015rate}. At the same time, the analysis of gradient flows driven by \emph{random} functionals is largely unexplored. Below, we mention some random models with a gradient structure related to the setting of this work.

Gradient flows in the space of probability measures driven by a random process (RQF is such a model) fall into the framework of stochastic evolution equations; well-posedness of these models has been established in \cite{coghi2019stochastic, gess2019stochastic}. In addition, RQF could be interpreted as the fast-jump limit of the corresponding piecewise-deterministic process. Thus, another related class of non-autonomous diffusion processes are piecewise-deterministic Markov processes introduced in \cite{davis1984piecewise}. In particular, depending on the jump rates and irreducibility of the processes the system may admit (unique or not) stationary measures, see e.g. \cite{bakhtin2012invariant, benaim2015qualitative}.

\paragraph{Synchronization in transformer models}
The joint dynamics of tokens in continuous-time models driven by the pure self-attention mechanism has been studied in  \cite{geshkovski2024emergence}, considering various settings depending on the parameters $Q, K$ and $V$. The structure of the critical points of the corresponding energy functional in the mean-field setting was characterized in \cite{burger2025analysis}. Convergence of tokens to a single cluster have later been proved in the simplified setting $Q= V = I$, $K = \beta I$  in \cite{geshkovski2024mathematical, criscitiello2024synchronization, polyanskiy2025synchronization}. The rates of convergence in the same one-parametric case have been studied in \cite{chen2025quantitative}. In the analogous setting, in \cite{alvarez2026perceptrons}, the stability of clustered states has been established in the presence of Feed-Forward layers with fixed parameters. Clustering in the special case of hardmax self-attention is proved in \cite{alcalde2025clustering, alcalde2025attention}.

Even in cases when any stable equilibrium state consists  necessarily of a singleton, the self-attention models have been shown to admit \emph{metastable} states consisting of multiple clusters, see \cite{geshkovski2024dynamic, bruno2024emergence, bruno2025multiscale}. For estimates on the number of metastable clusters,  see \cite{bruno2024emergence, geshkovski2024number}.

Several results also concern the behavior of tokens in transformers in the presence of randomness. In particular, if the tokens are driven by the deterministic self-attention drifts and independent Brownian diffusions, the resulting mean-field model is given by the McKean-Vlasov equation. The long-time behaviour of such a system is studied in \cite{shalova2026solutions, balasubramanian2025structure}. A more realistic model of noise has been recently considered in \cite{fedorov2026clustering}, where the noise is considered to be common for all tokens and comes from the randomization of the self-attention parameters. In particular, it was shown that the clustering phenomenon appears in the case when the parameter $V = V(t)$ is a matrix of independent Wiener processes.

\subsection{Main results}
\label{sec:intro-results}

In this work, we study distributional properties of the solutions to Eq.~\eqref{eq:rqf}, i.e., \emph{invariant measures} (see Section \ref{sec:main-distributional}), and give a path-wise characterization, describing the \emph{random attractor} of the RQF (see  Section \ref{sec:main-pw}). 

We argue that the behavior of the RQF can be understood as a random counterpart of the dynamics of the deterministic quadratic form. In order to compare the two systems, we first give an informal characterization of the deterministic quadratic form following \cite{mahony1996gradient}, which is made rigorous in Section \ref{sec:dqf}. 
\begin{theorem}[Deterministic Quadratic Form]
\label{th:intro-dqf}
    Let $M \in \Sym^n$ be a symmetric matrix, sampled from the Gaussian Orthogonal Ensemble (cf.~\eqref{eq:GOE}). 
    Then, with probability $1$, there exists $x^*\in \bbS^{n-1}$ such that the gradient flow of the quadratic form \eqref{eq:intro-dqf} satisfies 
    \[
    \lim_{t\to \infty}\min \left( \dist(x(t), x^*),  \dist(x(t), -x^*)\right) =0
    \]
    for a.e. initial condition $x_0 \in \bbS^{n-1}$.
    In other words, almost every trajectory of the gradient flow $x(t)$ converges to either $x^*$ or $-x^*$.
\end{theorem}
\begin{figure}
    \centering
    \includegraphics[width=0.32\linewidth]{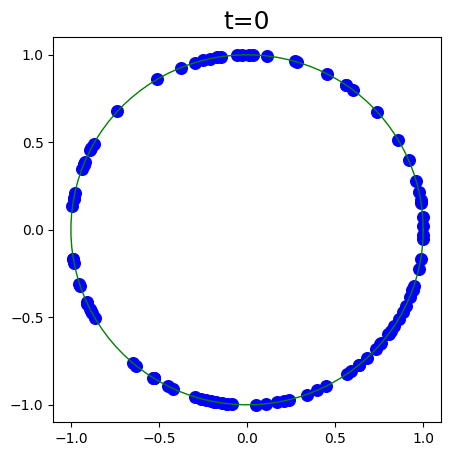}
    \includegraphics[width=0.32\linewidth]{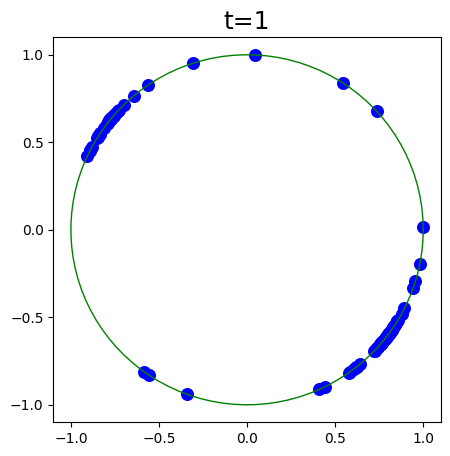}
    \includegraphics[width=0.32\linewidth]{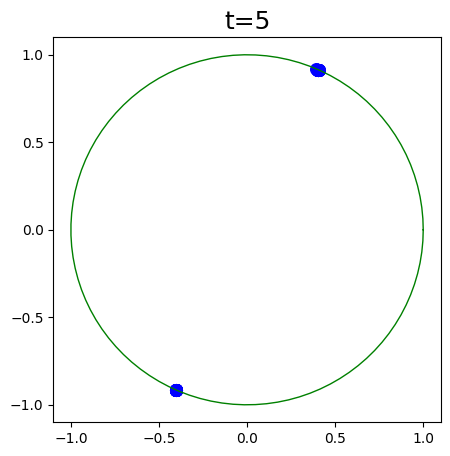}
    \caption{Ensemble of RQFs driven by the same process $Q_t$ from different initial conditions. At time $t \sim 5$ the trajectories approach the \emph{random attractor} consisting of two antipodal points that further move in time. }
    \label{fig:rqf}
\end{figure}
Since the points $x^*$ and $-x^*$ are opposite poles on the sphere, we say that any measure supported on two opposite points is an \emph{anti-polar} configuration. The main result of this work is informally summarized in the following theorem; for rigorous statements see Theorems \ref{th:rqf-one} and  \ref{th:rqf-attractor}.
\begin{theorem}[Random Quadratic Form]
\label{th:intro-main}
    Let $Q_t$ be the driving process as described in~\eqref{eq:A}. Let $X_t$ and $Y_t$ be the RQF processes driven by $Q_t$ with (possibly) different initial conditions $X_0, Y_0 \in \bbS^{n-1}$, then
    \begin{itemize}
        \item the law of $X_t$ (and $Y_t$) in the large time limit converges to the uniform measure on the sphere,
        \item for almost every $\omega$, the two RQF processes $X_t, Y_t$ satisfy
        \[
        \lim_{t\to \infty} \min \left(\dist(X_t, Y_t), \dist(X_t, -Y_t)\right) =0.
        \]
        In other words, $X_t$ and $Y_t$ either converge to each other (polar) or become opposite (anti-polar configuration).
    \end{itemize}
\end{theorem}
In other words, an arbitrary set of particles driven by the RQF converges to a polar/anti-polar configuration (see Figure \ref{fig:rqf}), and the poles of this configuration become uniformly distributed in the large time limit. We remark that the second part of Theorem \ref{th:intro-main} is also satisfied in the deterministic setting and, in that case, can be seen as a corollary of Theorem \ref{th:intro-dqf}. \emph{This illustrates the fact that the RQF is a consistent generalization of the DQF in terms of the synchronizing behavior, due to the persistent gradient structure}.

The rest of the paper is structured as follows. In Section \ref{sec:dqf} we give a rigorous description of the dynamics of the deterministic quadratic form. In Section \ref{sec:rds} we give an introduction to random dynamical system theory and synchronization by noise. In Section \ref{sec:main-distributional} we study invariant measures of the RQF and in Section \ref{sec:main-pw} the random attractor of the system. Finally, we outline potential generalizations and interpret the results from the machine learning perspective in Section \ref{sec:discussion}.

\paragraph{Acknowledgments.} AS is grateful to Mark A. Peletier and Anton Khvalyuk for many insightful discussions at the early stages of the project.
ME thanks Dennis Chemnitz, Robin Chemnitz and Michael Scheutzow for many insighful discussions on random attractors.
This work was supported by the Dutch Research Council (NWO), in the framework of the program VI.Vidi.233.133 `A Rigorous Framework for Transient Random Dynamics'.
\section{Deterministic Quadratic Form}
\label{sec:dqf}
\begin{figure}
    \centering
    \includegraphics[width=0.45\linewidth]{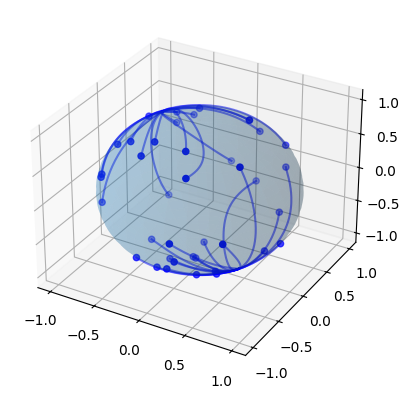}
    \includegraphics[width=0.5\linewidth]{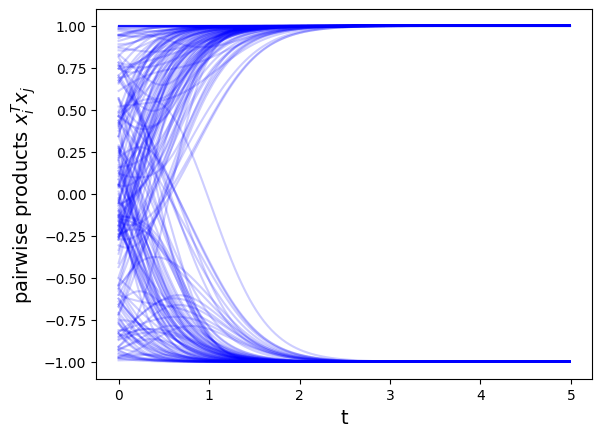}
    \caption{Solutions of the gradient flow of the deterministic quadratic form. On the left: trajectories with different initial conditions (marked by blue dots). On the right: dynamics of the pairwise scalar products between the trajectories.}
    \label{fig:dqf}
\end{figure}
Before we proceed to the analysis of the random model, we give a short description of the gradient flow dynamics of the deterministic quadratic form. In particular, let $M = M^T$ be an $n\times n$ real matrix with eigenvalues $\lambda_1 \geq \lambda_2 ... \geq \lambda_n$ and eigenvectors $\{v_i\}$. In addition, let $\Lambda_m =\text{span} \{v_1 ... v_k: \lambda_k = \lambda_1\}$ be the eigensubspace corresponding to the maximal eigenvalue. We consider the gradient flow dynamics of the quadratic functional on the unit sphere $\bbS^{n-1}$, namely
\begin{equation}
\label{eq:det-qf}
\dot x = - \nabla \left(\frac12x^TMx \right) = - P_xMx, \qquad x(0) = x_0.
\end{equation}
The system arises in the context of the principal component analysis and has been studied in, e.g., \cite{mahony1996gradient}.  It has been shown that for $\vol_{\bbS^{n-1}}$-a.e.\footnote{volume measure of the sphere in the standard topology}~initial condition $x_0$, in the large time limit the solution converges to the maximal eigensubspace, i.e., $\lim_{t\to \infty}\dist(x(t), \Lambda_m \cap \bbS^{n-1}) =0$. In particular, the following results hold:
\begin{theorem}[Deterministic Quadratic Form, {\cite[Theorem 2.1]{mahony1996gradient}}]
\label{th:det-qf}
     Let $M=M^T$ be a real symmetric $n \times n$ matrix with the maximal eigensubspace $\Lambda_m$. Then the gradient flow dynamics of the quadratic functional generated by $M$ on a sphere \eqref{eq:det-qf} has the following properties:
     \begin{itemize}
         \item the set of critical points is given by
\[
\Crit_M :=\{x\in \bbS^{n-1}: Mx = \lambda x,\  \lambda \in \{\lambda_1, \dots, \lambda_n \} \}.
\]
\item the exact solution to \eqref{eq:det-qf} is given by
\[
x(t)=e^{t M} x_0 /\left\|e^{t M} x_0\right\|,
\]
\item for all initial conditions $x_0 \in \bbS^{n-1}$, the solution $x(t)$ of the gradient flow \eqref{eq:det-qf}
exists and converges to a vector $x_{\infty} \in \bbS^{n-1}$ as $t \rightarrow \infty$. Moreover, for almost all initial conditions the solution $x(t)$ converges exponentially fast to some $x_\infty \in \Lambda_m \cap \bbS^{n-1}$.
\end{itemize}
\end{theorem}

We illustrate the behavior of the gradient flow solutions on Figure \ref{fig:dqf}.

\begin{remark}[$\dim \Lambda_m = 1$] Note that for a matrix $M$ sampled from the Gaussian Orthogonal Ensemble\footnote{Gaussian Orthogonal Ensemble (GOE) is the probability distribution over the space of symmetric real matrices, defined as
\begin{equation}
    \label{eq:GOE}
    M \sim \text{GOE} \ \Leftrightarrow M_{ii} \sim \calN(0, 2),\  M_{ij} \sim \calN(0,1),
\end{equation}
and all entries $M_{i\leq j}$ are independent.}, $\dim \Lambda_m = 1$ with probability one. 
At the same time, when 
$\dim \Lambda_m = 1$, the attractor of the system, namely the set $\Lambda_m \cap \bbS^{n-1}$ consists of just two antipodal points. In the following sections we argue that a similar behavior is exhibited when $A$ is not a fixed matrix but a random process. 

We also remark that the increments of the noise process $Q_t$ are (scaled) GOE matrices, implying that w.p. $1$ the set of minimizers of the random driving functional consists of two anti-polar points. Therefore, the convergence to an anti-polar configuration is natural in the random case.
\end{remark}
\begin{remark}[Wasserstein gradient flow]
If instead of a single particle $x$ we are given an initial probability distribution $\rho_0 \in \calP(\bbS^{n-1})$ where all the points move according the dynamics \eqref{eq:det-qf}, then the evolution on $\rho$ is given by the \emph{Wasserstein gradient flow} of the potential energy 
\[
\calF_M(\rho) :=\int_{\bbS^{n-1}} x^TMx\ \rmd \rho(x),
\]
namely the PDE of the form:
\begin{equation}
\label{eq:wgf}
\partial_t\rho =-\nabla^{W_2} \calF_M(\rho) %
\end{equation}
where $\nabla^{W_2} \calF_M(\rho) = \nabla \cdot (\rho \nabla \frac{\delta \calF_M}{\delta \rho})$ is often called the \emph{Wasserstein gradient} of $\calF_M$ and $\frac{\delta \calF_M}{\delta \rho}$ denotes the first variation of $\calF_M$. %
The PDE \eqref{eq:wgf} can be interpreted as a weak form of the gradient flow \eqref{eq:det-qf} and, as follows from Theorem~\ref{th:det-qf},
 the critical points of \eqref{eq:wgf} are all possible measures $\tilde \rho$ supported on the set $\Crit_M$. Stationarity of such measures is transparent since the Wasserstein gradient takes the following form (in a weak sense):
 \[
 \nabla^{W_2} \calF_M(\rho)(x) = -\nabla \cdot (\rho(x) Mx  )
 \]
 and thus $\nabla^{W_2} \calF_M(\rho)= 0$ on every measure $\rho$ supported on $\Crit_M$.
\end{remark}
\begin{remark}[Random Wasserstein gradient flow]
\label{rem:random-wgf}
    Extending the formal calculation from the introduction to the Wasserstein spaces of probability measures one can also consider a (degenerate) stochastic PDE of the form
    \begin{equation*}
\partial_t\rho = \nabla \cdot (\rho(x) \rmd Q_t x),
\end{equation*}
as the distributional counterpart of the RQF. We call the model degenerate since the noisy process is the same at every $x\in \bbS^{n-1}$, instead of the classical case of the driving process being the space-time white-noise. We also remark that well-posedness of models of this type was established in \cite{coghi2019stochastic, gess2019stochastic}.

We also argue that randomizing the driving functional is one of the natural ways to introduce randomness into the framework of Wasserstein gradient flows and thus, the RQF model can be seen as an elementary example of a class of `random Wasserstein gradient flows', defined as
\[
    \partial_t\rho = -\nabla^{W_2}  \calF(\rho, \theta_t),
\]
where $\theta$ is the noisy variable. Under certain regularity and measurability conditions, such a dynamics can be studied using the random dynamical systems framework, which we introduce in the following section.
\end{remark}
\section{Background on Random Dynamical Systems}
\label{sec:rds}
\subsection{RDS and SDEs}
We introduce the notion of a (global) continuous random dynamical system in two-sided continuous time $\mathbb R$ (see, e.g., the canonical reference \cite{Arnold98}). Since our analysis is focused on the compact Riemannian manifold $\mathbb S^{n}$, we do not need to regard issues of finite-time blow-up that occur, e.g., in $\mathbb R^n$. We will work with the following definition of an RDS on a Polish space $\calX$ (i.e., complete, separable metric space) and then use the manifold structure when necessary:
\begin{definition}[Random dynamical system (RDS)] Let $(\Omega, \calF, \bbP)$ be an abstract probability space and $(\calX, \calB(\calX))$ be a Polish space (with the corresponding $\sigma$-algebra). An RDS consists of two parts: 
\begin{enumerate}
    \item[(a)] The model of the noise in form of a family of measurable maps $(\theta_t: \Omega \to \Omega)_{t \in \mathbb R}$, satisfying 
\begin{itemize}
    \item $\theta_0\omega = \omega, \ \forall \omega\in \Omega$, \  $\theta_{t+s}\omega = \theta_{t}\theta_{s}\omega, \ \forall t, s \in \bbR, \omega \in \Omega$,
    \item $\bbP(A) = \bbP(\theta^{-1}_t A), \ \forall t\in\bbR, A\in \calF$,
\end{itemize}
also called \emph{metric dynamical system} $((\theta_t)_{t \in \mathbb R}, \Omega, \mathbb P)$.

\item[(b)] The model of the dynamics, a measurable map $\varphi: \bbR \times \Omega \times \calX \to \calX$ which, for all $\omega \in \Omega$ and $x \in \mathcal X$ satisfies the \emph{cocycle property} 
\begin{equation} \label{eq:cocycle}
    \varphi(0, \omega, x) = x, \ \text{ and } \ \varphi(t+s, \omega, x) = \varphi(t, \theta_s\omega, \varphi(s, \omega, x)), \ \forall s, t \in \bbR.
\end{equation} 
\end{enumerate}
The RDS is called continuous if the map $\varphi(\cdot, \omega, \cdot): \bbR \times  \calX \to \calX$ is continuous for all $\omega \in \Omega$. 

\noindent  If $\cal X$ is a Riemannian manifold, the RDS is called $C^k$ if the map $\varphi(t, \omega, \cdot): \calX \to \calX$ is $C^k$ for all $\omega \in \Omega$, $t \in \mathbb R$. 
\end{definition}

Going back to \cite{kunita1990stochastic, arnold1995perfect}, it has been shown that a Stratonovich SDE with regular enough coefficients generates a unique RDS. The underlying noise model is given as follows: consider the space $\Omega=C_0\left(\mathbb{R}, \mathbb{R}^m\right)$ of all continuous functions $\omega: \mathbb{R} \rightarrow \mathbb{R}^m$ satisfying $\omega(0)=0$. We define the metric $d$ on $\Omega$ as
\[
d(\omega, \widehat{\omega}):=\sum_{n=1}^{\infty} \frac{1}{2^n} \frac{\|\omega-\widehat{\omega}\|_n}{1+\|\omega-\widehat{\omega}\|_n}, \quad \|\omega-\widehat{\omega}\|_n:=\sup _{|t| \leq n}\|\omega(t)-\widehat{\omega}(t)\|,
\]
and let $\mathcal{F}=\mathcal{B}(\Omega)$ be the Borel $\sigma$-algebra on $(\Omega, d)$. Then the Wiener measure corresponding to the $m$-dimensional Wiener process $\left(W_t^1(\omega), \ldots, W_t^m(\omega)\right)^{\top}:=\omega(t)$ is given by
\[
\mathbb{P}\left(\left\{\omega \in \Omega: \omega_1(t) \leq x_1, \ldots, \omega_m(t) \leq x_m\right\}\right)=\frac{1}{(2 \pi t)^{d / 2}} \int_{-\infty}^{x_1} \cdots \int_{-\infty}^{x_m} e^{-\|y\|^2 / 2|t|} \mathrm{d} y_1 \cdots \mathrm{~d} y_m,
\]
for all $x \in \mathbb{R}^m$.
The family of shifts $\left(\theta_t\right)_{t \in \bbR}$ on the probability space $(\Omega, \mathcal{F})$, given by
\begin{equation}
\label{eq:shift}
    \theta_t \omega(\cdot):=\omega(t+\cdot)-\omega(t),
\end{equation}
is measure preserving (and even ergodic), and forms the metric dynamical system $((\theta_t), \Omega, \mathbb P)$

One can then formulate the following statement as a direct corollary of \cite[Theorem 2.3.42]{Arnold98}, assumuing that $\calX$ is a compact $n$-dimensional Riemannian manifold (as, e.g., $\mathbb S^{n}$).

\begin{proposition}[SDE as an RDS] \label{prop:SDE}
Consider the Stratonovich SDE of the form
\begin{equation}
\label{eq:sde}
\rmd X_t=F_0\left(X_t\right) \rmd t+\sum_{j=1}^m F_j\left(X_t\right) \partial W_t^j, \quad X_0=x \in \calX, t \in \mathbb{R},
\end{equation}
where $F_0 \in C^{k, \delta}(\calX, T\calX)$ and every $F_j \in C^{k+1, \delta}(\calX, T\calX)$  for some $k \in \mathbb{N}$ and $\delta \in(0,1]$. 
Then there exists a unique measurable function $\varphi(t, \omega, x)$ such that
\begin{itemize}
\item $\varphi(t, \cdot, x)$ is the solution of the SDE \eqref{eq:sde},
    \item $(\theta, \varphi)$ is a $C^k $ RDS with $\theta$ as in Eq.~\eqref{eq:shift}.
\end{itemize}
\end{proposition}

For any RDS, one may consider the corresponding deterministic dynamical system $(\Theta_t)_{t \in \mathbb R}$ on $\Omega \times  \calX$, called \emph{skew product flow}. For every $t \in \mathbb R$ the map $\Theta_t: \Omega \times \calX \to  \Omega \times \calX$
is given by
$$\Theta_t(\omega, x) := (\theta_t \omega, \varphi(t, \omega, x)).$$
Indeed the flow properties $\Theta_0 = \Id_{\Omega \times \calX}$ and $\Theta_{s+t}(\omega,x) = \Theta_s(\Theta_t(\omega,x))$, whenever all terms are well-defined, can be deduced from the cocycle property \eqref{eq:cocycle}.

Now, for any $u,v \in \mathbb R$ with $u < v$, we denote by $\mathcal{F}_u^v \subset \mathcal F$ the sub-$\sigma$-algebra generated by  the random variable $\varphi(t, \theta_s \omega, x)$ for $x \in \calX$ and $t, s \in \mathbb R$ with $u\leq s\leq v$ and $0< t \leq v-s$, (and the subsets of $\mathcal F$ of zero measure).

This means, in particular, that $\varphi(t, \theta_s \omega, x)$ is $\mathcal F_s^{s+t}$-measurable. We define the $\sigma$-algebras 
\begin{align*}
\mathcal F_{-\infty}^v &= \sigma( \mathcal F_u^v \,:\, u \in \mathbb R, u < v ),\\
\mathcal F_{u}^{\infty} &= \sigma( \mathcal F_u^v \,:\, v \in \mathbb R, u < v ),\\
\mathcal F_{-\infty}^{\infty} &= \sigma( \mathcal F_u^v \,:\, v, u \in \mathbb R, u < v )
\end{align*}
\begin{definition}
The RDS $(\theta, \varphi)$ is said to be \emph{Markov} if $\mathcal F_{-\infty}^0$ and $\mathcal{F}_0^{\infty}$ are independent, i.e., if past and future are independent.
\end{definition}
For such random dynamical systems, one may  introduce the transition function as
\begin{equation}\label{eq:transition_RDS}
\hat P_t(x, B) = \mathbb{P} (\{ \varphi(t, \cdot,x) \in B \}), \ x \in X, \ B \in \mathcal B(\calX), \ t \in \mathbb R_+.
\end{equation}
It is now straightforward to construct a Markov process out of the Markov RDS in a canonical way. For SDEs of the form \eqref{eq:sde}, these are exactly the solution processes. In the following, we will focus on Markov RDS.

\subsection{Random attractors and sample measures}
\label{sec:attractors}

The RDS framework allows to study the long-time behavior of the underlying system for each given noise realization. 
In particular, one of the key objects in the theory of RDS  are \emph{random attractors}: random sets attracting the trajectories of the dynamical system.

In more detail, a map $A: \Omega \to \mathcal \calB(\calX)$ is called a \emph{random compact set}, if $A(\omega)$ is almost surely a non-empty compact set and for each $x \in \mathbb \cal X$ the map $\dist(x, A(\cdot)) \coloneqq \inf_{y \in A(\cdot)} d(x,y)$ is $\mathcal F$-measurable (see, e.g., \cite[Definition 1.6.1]{Arnold98}). It is furthermore called \emph{measurable with respect to the past}, if the map $\dist(x, A(\cdot))$ is $\mathcal F_{-\infty,0}$-measurable. A random compact set is called $\varphi$\emph{-invariant}, if  for all $t\geq 0$
$$\varphi(t, \omega, A(\omega)) = A(\theta_t(\omega)).$$
Here, we will focus on \emph{point attractors} as opposed to \emph{set attractors} since we are interested in the behavior of the $N$-point motion as opposed to topological properties of sets.

    \begin{definition}[Random Point Attractor]\label{def:Attr}
   A random compact $\varphi$-invariant set $(A(\omega))_{\omega \in \Omega}$ is called
	\begin{enumerate}
		\item[(i)] \emph{pull-back point attractor}, if it is measurable with respect to the past and for every $x \in \cal X$
            $$\dist\big(\varphi(t, \theta_{-t}\omega, x), A(\omega)\big) \to 0, \faa \omega \in \Omega.$$
		\item[(ii)] \emph{forward point attractor}, if for every $x \in \cal X$
            $$\dist\big(\varphi(t, \omega, x), A(\theta_t \omega)\big) \to 0, \faa \omega \in \Omega.$$
	\end{enumerate}
 Replacing the almost-sure convergence with convergence in probability yields the definition of a \emph{weak point attractor}; note that weak pull-back and forward attraction are equivalent due to the invariance of $\mathbb P$ under $\theta_t$.
\end{definition}

The existence of a weak point attractor, which is the form of random attractor we focus on unless stated otherwise, is now closely linked to the support of invariant measures for the Markov RDS.
The \emph{correspondence theorem} (cf., e.g., \cite[Theorem 4.2.9.]{KuksinShirikyan12}) establishes a one-to-one correspondence between stationary distributions $\rho$ of the induced Markov process and invariant measures $\mu$ of the skew product flow whose disintegrations are measurable with respect to the past $\mathcal F_{-\infty, 0}$, called \emph{Markov measures}. 
In the following, we call a random variable with a certain property \emph{essentially unique}, if it coincides with every other random variable with that property up to a set of zero measure, and denote the set of probability measures on $\Omega \times \cal X$ by $ \calP(\Omega \times \cal X)$. We give the following short version of the correspondence theorem:
\begin{theorem}[Correspondence Theorem]\label{theo:Correspondence}
Let $(\theta, \varphi)$ be a Markov RDS with unique stationary measure $\rho$. Then the weak limit 
$$ \varphi(t, \theta_{-t}\omega, \cdot)^* \rho \stackrel{w}{\to} \mu_\omega,  \ \text{ as } t \to \infty,$$
exists almost surely and $\mu_\omega$ is the essentially unique $\mathcal F_{-\infty, 0}$-measurable random measure, such that the Markov measure $ \mu \in \calP(\Omega \times \cal X)$ given by
$$ \mu(\dd \omega, \dd x) = \mathbb P(\dd \omega) \mu_\omega(\dd x)$$
is $\Theta$-invariant. The $\cal X$-marginal of $\mu$ is given by  $\mathbb E[\mu_\omega] = \rho.$ 
\end{theorem}

The random measures $(\mu_\omega)_{\omega \in \Omega}$ are called \emph{sample measures} or \emph{statistical equilibria}.
As an immediate consequence of the $\Theta$-invariance of $ \mu$, the sample measures must be almost surely invariant, i.e. for every $t \in \mathbb R$ and almost every $\omega \in \Omega$ we have
$$\varphi(t, \omega, \cdot)^*\mu_\omega =  \mu_{\theta_t\omega}.$$

We now have the following characterization of a weak point attractor via these measures:

\begin{proposition}[Existence of a Weak Point Attractor]\label{prop:ExWeakPntAttr}
 Let $(\mu_\omega)_{\omega \in \Omega}$ be the sample measures defined in Theorem \ref{theo:Correspondence}. Then, if the random set 
 $$A(\omega) := \operatorname{supp}(\mu_\omega)$$
 is compact, it is a weak point attractor and it is minimal, i.e., any weak point attractor $\tilde A(\omega)$ satisfies $\tilde A(\omega) \supseteq A(\omega)$ almost surely.
\end{proposition}
\begin{proof}
  This follows directly from \cite[Proposition 2.20]{FGS1}.
\end{proof}
Note that compactness of $\operatorname{supp}(\mu_\omega)$ follows trivially if we are on a compact metric space like $\mathbb S^n$. 

In general, the statistical equilibrium of a random dynamical system is either fully discrete or contains no point masses. In particular, as follows from \cite{LeJan1987}, we can state (see also \cite[Proposition 2.19]{FGS2}):
\begin{proposition}[Discrete vs.~continuous $\mu_\omega$]
\label{prop:discrete_cont}
   The sample measures $\mu_\omega$ are either $\Omega$-a.s. continuous, i.e., satisfy $\mu_\omega(\{x\}) = 0$ for every $x\in\calX$, or are $\Omega$-a.s.~supported on a finite number of $\mathcal F_{- \infty}^0$-measurable random points $\{a_1(\omega), \dots, a_N(\omega)\}$ and given by
\begin{equation}
   \label{eq:discrete-mu} 
\mu_{\omega} = \frac{1}{N} \sum_{i=1}^N \delta_{a_i(\omega)}.
\end{equation} 
\end{proposition}

The properties of the sample measures $\mu_\omega$ are closely related to the stability of the corresponding dynamics. In the case of a differentiable structure on $\calX$, stability can be studied via \emph{Lyapunov exponents}, which we introduce in the next section.

\subsection{Lyapunov exponents and discrete random attractor}

Let us now assume that $\calX$ is a compact Riemannian manifold with tangent bundle $T \calX$.
In the following, we write $\langle, \rangle_x$ for the metric on the tangent space $T_x \cal X$ at $x$ and  $\| \cdot \|_x$ for the corresponding norm, ignoring $x$ in the notation when it is clear from the context.
For a vector field $f: \calX \to T \calX$, we write $\rmD f$ for the covariant derivative, considered for fixed $x \in \calX$ as a linear operator $\rmD f(x) : T_x \calX \to T_x \calX$. 

Assuming the situation of Proposition~\ref{prop:SDE} with $k \geq 1$, we can then consider the extended variational process $(X_t, Y_t)$ on $T \cal X$ given by~\eqref{eq:sde} and
\begin{equation}
\label{eq:var_process}
\rmd Y_t = \rmD F_0(X_t) Y_t\,\rmd t + \sum_{j=1}^m \rmD F_j(X_t) Y_t \partial  W_t^j\,,  \ Y_0 \in T_x M\, .
\end{equation}

For all $X_0 \in \cal X$ and $Y_0 \neq 0$, we consider the characteristic exponent
\begin{align*}
\Lambda(X_0, Y_0) = \limsup_{t \to \infty} \frac{1}{t} \log  \|Y_t\|.
\end{align*}
It is by now classical (see, e.g., \cite[Section 2.1]{BlessingBlumenthalBredenEngel2025}) that if the Markov solution process of the SDE~\eqref{eq:sde} has a unique invariant measure $\rho$, then, by ergodicity, this characteristic exponent is independent from the inital conditions and we obtain the (maximal) \emph{Lyapunov exponent} $\Lambda$ as the almost sure limit
\begin{align}\label{eq:asymptoticLE2} 
\Lambda = \lim_{t \to \infty} \frac{1}{t} \log  \|Y_t\|
\end{align}
for almost all $(X_0, Y_0)$. The sign of this exponent helps to determine the asymptotic behavior of the random dynamical system associated to the SDE~\eqref{eq:sde}. Note that $\rho$ corresponds with a unique invariant Markov measure $\mu$ by Theorem~\ref{theo:Correspondence}.%

\begin{proposition}[Discrete random attractor%
]
\label{prop:discrete}
Assume that the $C^1$ RDS $(\theta, \varphi)$ is induced by an SDE~\eqref{eq:sde} with unique stationary measure $\rho$ and maximal Lyapunov exponent $\Lambda < 0$. Then the minimal weak random point attractor $A(\omega) = \operatorname{supp}(\mu_\omega)$ consists of finitely many $\mathcal F_{- \infty}^0$-measurable random points and the sample measures are given by Eq. \eqref{eq:discrete-mu}.
\end{proposition}

\begin{proof}
    This can be seen from \cite[Lemma 2.19]{FGS1} in combination with \cite[Corollary 3.4]{FGS1}, where a local stable manifold theorem is used on $\mathbb R^n$ (directly transferrable to our compact situation). 

    The statement is also a direct corollary of \cite[Proposition 2]{LeJan1987}, as discussed in \cite{baxendale1991statistical}.
\end{proof}

\section{Main Results}
Our results are two-fold. In Section \ref{sec:main-distributional}, we characterize the distributional properties of the solutions of the RQF and a coupled system of two RQFs. In particular, we derive the correspondning Fokker-Planck equations and show that: 
\begin{itemize}
    \item RQF is a Brownian motion, see Theorem~\ref{th:rqf-one},
    \item coupled RQFs admit clustered invariant measures, see Theorem \ref{th:rqf-two}.
\end{itemize}

In Section \ref{sec:main-pw} we apply results from RDS theory introduced in Section \ref{sec:rds} to give an $\omega$-wise characterization of solutions of the RQF. In particular, in Theorem \ref{th:rqf-attractor} we show that the random attractor of the RQF almost surely consists of two anti-polar points and, hence, that the RQF exhibits a specific form of synchronizing behavior. 
\subsection{Invariant measures}
\label{sec:main-distributional}
For a general diffusion process $X_t$, its dynamical properties may be characterized by the infinitesimal generator:
\begin{definition}[Infinitesimal generator] Let $X_t$ be a Feller process on a Banach space~$\calX$. Then the operator $L: C_c^\infty(\calX) \to C_c^\infty(\calX)$ defined by
\[
(Lf)(x) := \lim_{t\to 0}\frac{1}{t} \bbE \left[f(X_t) - f(X_0)\big| X_0 = x\right]
\]
 is called the \emph{infinitesimal generator} of $X_t$.
\end{definition}
In particular, the evolution of the law of $X_t$ is given by the corresponding Fokker-Planck equation
\[
\partial_t f - L^*f =0,
\]
where $L^*$ is the (formal) $L^2$ adjoint of the generator $L$. The stationary measures, i.e., solutions of the stationary Fokker-Planck equation $L^*\rho = 0$, can be defined in terms of $L$:
\begin{definition}[Invariant measure]
A probability measure $\rho \in \calP(\calX)$ is an invariant measure of the diffusion process $X_t$ if  
\[
\int_{\calX} (L u)(x) \rm\rmd \rho(x) = 0
\]
holds for all $u\in C^\infty_c(\calX)$.
\end{definition}
We first show that RQF is generated by the Laplace-Beltrami operator on a sphere (with the natural topology) and thus is a Brownian motion. With a slight abuse of notation, we always identify the measure with its density with respect to the volume measure on $\bbS^{n-1}$ whenever the density exists.
\begin{theorem}[RQF is a Brownian motion]
\label{th:rqf-one}
Let $\rho_t = \law (X_t)$. Then $\rho_t: (0, \infty) \to C^\infty(\bbS^{n-1})$ is the unique (classical) solution of the rescaled heat equation, i.e.,
\[
\partial_t \rho_t - \frac{1}{2} \Delta \rho_t = 0, \quad \rho_t \stackrel{w}{\to} \law (X_0) \text{ as } t \downarrow 0,
\] 
where $\Delta$ is the Laplace-Beltrami operator on $\bbS^{n-1}$. In particular, the only invariant measure of the RQF is the uniform measure on the sphere $\bar \rho = \Gamma(\frac{n}{2})(2\pi^{n/2})^{-1} \vol_{\bbS^{n-1}}$, where $\Gamma$ denotes the gamma function.
\end{theorem}
\begin{proof}
The proof is straightforward and is based on the comparison with the conventional definition of Brownian motion on a sphere. In particular, consider the process defined by the Stratonovich SDE
\begin{equation}
\label{eq:bm}
\rmd W_t = P_{W_t} \partial B_t,
\end{equation}
where $B_t$ is an $n$-dimensional Wiener process. The generator of $W_t$ is $\frac{1}{2}\Delta$ and, since $\Delta$ is essentially self-adjoint on $\bbS^{n-1}$, the statement of the Theorem holds for $W_t$. We thus only need to show that $\law(X_t)$ solves the Fokker-Planck equation defined by $W_t$.

     Rewriting \eqref{eq:bm} in Euclidean coordinates we get
     \begin{equation}
         \label{eq:bm-euc}
    \rmd W^i_t = \sum_{j} \hat V^i_{j} (W_t) \partial B^j, \quad \text{where} \quad \hat V^i_{j} (W_t) = P_{W_t}^{ij} =  \delta_{i, j} - W_t^i W_t^j.
     \end{equation}
     Using \cite[Section 4.3.6]{gardiner2004handbook}, we conclude that, with this notation, the corresponding Fokker-Planck equation takes the  form 
     \[
     \partial_t f = L^* f = \frac{1}{2}\sum_{i, j, l} \partial_i\left(\hat V^i_{j} \partial_l \left(\hat V^l_{j} f\right)\right).
     \]

     The analogous reformulation of the RQF process
     \eqref{eq:rqf} reads as
    \begin{equation}
    \label{eq:brownian-1}
    \rmd X^i_t = -\sum_{j, k} V^i_{j, k} \partial Q^{j, k} = -\frac{1}{\sqrt 2}  \sum_{j, k} (V^i_{j, k} + V^i_{k, j}) \partial B^{j, k} , \quad \text{where} \quad V^i_{j, k} = (\delta_{i, j} - X_t^i X_t^j)X^k_t,
   \end{equation}
    implying that the Fokker-Planck equation corresponding to the RQF takes the form 
    \begin{equation}
    \label{eq:brownian-2}
    \partial_t f = L^*_X f = \frac{1}{2}\sum_{i, j, k, l} \partial_i\left( V^i_{j, k} \partial_l \left(V^l_{j, k} f\right)\right) + \frac{1}{2}\sum_{i, j, k, l} \partial_i\left( V^i_{j, k} \partial_l \left(V^l_{k, j} f\right)\right) =: L^d f + L^c f,
    \end{equation}
    where $L^d$ stands for the diagonal and $L^c$ for the cross-diagonal part.
    We will show that $L^*_X =L^d = L^* $. 
    
    Plugging the expression for $V^i_{j, k}$ from \eqref{eq:brownian-1} into \eqref{eq:brownian-2} for the diagonal part $L^d $ we calculate for arbitrary $\rho \in C^2(\bbS^{n-1})$ and $x = (x^1, \dots, x^{n}) \in \bbS^{n-1}$:
    \begin{align*}
        (L^d \rho)(x) &= \frac12\sum_{i, j, k, l}\partial_i\left( \hat V^i_{j}x^k \partial_l \left(\hat V^l_{j}x^k \rho\right)\right) \\
        &= \frac12\sum_{i, j, k, l} \partial_i\left( \hat V^i_{j}\left(x^k\right)^2\partial_l \left(\hat V^l_{j} \rho\right)\right)  + \partial_i\left( \hat V^i_{j}x^j \hat V^l_{j} \rho \partial_l x^k\right) \\
        &= \frac12\sum_{i, j, l} \partial_i\left(\hat V^i_{j} \partial_l \left(\hat V^l_{j} \rho\right)\right)  + \frac12\sum_{i, j, k, l} \partial_i\left( \delta_{l, k}\hat  V^i_{j}x^k \hat V^l_{j} \rho \right) \\
        &= (L^* \rho)(x) + \sum_{i, j, l} \partial_i\left( \hat  V^i_{j}x^l (\delta_{j, l} - x^lx^j) \rho \right) \\
        &= (L^* \rho)(x) + \sum_{i, j} \partial_i\left( \hat  V^i_{j} \left(\sum_l \delta_{j, l} x^l  - \sum_l (x^l)^2 x^j\right) \rho \right) \\
        &=(L^* \rho)(x) + \sum_{i, j} \partial_i\left( \hat  V^i_{j} \left(x^j  -  x^j\right) \rho \right) = (L^* \rho)(x),
    \end{align*}
    and for the cross term $L^c $:
    \begin{align*}
        (L^c \rho)(x) &= \frac12\sum_{i, j, k, l}\partial_i\left( \hat V^i_{j}x^k \partial_l \left(\hat V^l_{k}x^j \rho\right)\right) \\
        &= \frac12\sum_{i, j, k, l} \partial_i\left( \hat V^i_{j} x^kx^j\partial_l \left(\hat V^l_{k} \rho\right)\right)  + \partial_i\left( \hat V^i_{j}x^k \hat V^l_{k} \rho\partial_l x^j\right) \\
        &= \frac12\sum_{i, k, l} \partial_i\left(x^k \partial_l \left(\hat V^l_{k} \rho\right) \sum_j\hat V^i_{j} x^j\right)  + \frac12\sum_{i, j, k, l} \partial_i\left( \delta_{l, j}\hat  V^i_{j}x^k \hat V^l_{k} \rho \right) \\
        &= 0 +\frac12\sum_{i, k, l} \partial_i\left( \hat  V^i_{l}x^k \hat V^l_{k} \rho \right) = %
         0,
    \end{align*}
    where we used the properties of the projection matrix from Lemma \ref{lemma:prop-proj} below and the fact that $\sum (x^i)^2 = 1$ on $\bbS^{n-1}$. Thus, we have $L^*_X = L^* = \frac{1}{2}\Delta $ and the result follows.
\end{proof}
\begin{lemma}[Properties of the diffusion matrix] 
\label{lemma:prop-proj}
For any $x \in \bbS^{n-1}$ let $\hat V^i_{j}$ be the elements of the projection matrix as in Eq. \eqref{eq:bm-euc}, then the following holds for all $i, k = 1\dots n$:
\begin{itemize}
    \item $P^2 = P$:
    \[
    \hat V^i_{k} = \sum_l \left(\hat V^i_{l} \hat V^l_{k} \right) 
    \]
\item $Px = 0$:
    \[
    \sum_{ k}  \hat  V^i_{k}  x^k =0. 
    \]
    \end{itemize}
\end{lemma}
\begin{proof}
    This follows directly from the definition of the projection matrix on $\bbS^{n-1}$.
\end{proof}
\paragraph{Two point process.} We now consider the dynamics of two RQF processes driven by the same noise process $Q_t$,
\begin{equation}
\begin{split}
\label{eq:rqf-two}
    \rmd X_t  
    &= -P_{X_t}  \partial Q_t  X_t, \quad X_0 \in \bbS^{n-1}, \\
    \rmd Y_t 
    &= -P_{Y_t}  \partial Q_t  Y_t, \quad \ \  Y_0 \in \bbS^{n-1}.
\end{split}
\end{equation}
While both $X_t$ and $Y_t$ are Brownian motions and thus (by the ergodic theorem) become uniformly distributed on the sphere in the large time limit, their mutual configuration behaves non-trivially due to the coupling by the noisy process.

In this section we study invariant measures $\rho \in \calP(\bbS^{n-1} \times \bbS^{n-1})$ of the coupled process~\eqref{eq:rqf-two}. By Theorem \ref{th:rqf-one}, both marginals of any stationary distribution are necessarily equal to the uniform measure $\bar \rho$. At the same time, as will be shown in Corollary~\ref{th:rqf-two}, both the polar $X_t = Y_t$ and the anti-polar configuration $X_t = -Y_t$ become invariant under the dynamics due to the common noise and thus the system admits a whole family of non-trivial (clustered) invariant distributions. %

Such a behavior stems from  the underlying symmetry and holds for a larger class of stochastic processes:
\begin{theorem}[Invariant measures of a two-point process]
\label{th:main-sym}
Let $X_t$ be a stochastic process on $\bbS^{n-1}$ of the form
\begin{equation}
\label{eq:sym}
\rmd X_t = F_0(X_t)\rmd t + \sum_{i=1}^k F_i(X_t) \partial B^i_t, \qquad X_0 \in \bbS^{n-1}
\end{equation}
where $B^i_t$ are independent Brownian motions in $\bbR$ and $F_i: \bbS^{n-1} \to  T\bbS^{n-1} \subseteq \bbR^n$ are tangent vector fields on the sphere satisfying $F_i(-X_t) = -F_i(X_t)$ in Euclidean coordinates. Then, any invariant measure $\hat\rho \in \calP(\bbS^{n-1})$ of the process $X_t$ generates a  family $(\rho_\alpha)_{\alpha \in [0,1]} \subset \calP(\bbS^{n-1} \times \bbS^{n-1})$ given by
\[
\rho_\alpha(\rmd x, \rmd y) = \hat \rho(\rmd x) \times \left(\alpha \delta_x(\rmd y) + (1-\alpha)\delta_{-x}(\rmd y)\right), \quad \alpha \in [0,1],
\]
such that every $\rho_\alpha$ is an invariant measure of the coupled two point process
\begin{equation}
\begin{split}
\label{eq:rqf-two_general}
    \rmd X_t  
    &= F_0(X_t)\rmd t + \sum F_i(X_t) \partial B^i_t, \quad X_0 \in \bbS^{n-1},\\
    \rmd Y_t 
    &= F_0(Y_t)\rmd t + \sum F_i(Y_t) \partial B^i_t, \quad \ \  Y_0 \in \bbS^{n-1}.
\end{split}
\end{equation}
\end{theorem}

Applying the result to the RQF model \eqref{eq:rqf-two} we conclude that:
\begin{corollary}[Invariant measures of coupled RQFs]
\label{th:rqf-two}
Every measure $\rho_\alpha \in \calP(\bbS^{n-1} \times \bbS^{n-1})$ of the form
\[
\rho_\alpha(\rmd x, \rmd y) = \bar \rho(\rmd x) \times \left(\alpha \delta_x(\rmd y) + (1-\alpha)\delta_{-x}(\rmd y)\right), \quad \alpha \in [0,1]
\]
is an invariant measure of the coupled process \eqref{eq:rqf-two}.
\end{corollary}
\begin{proof}
    Using the fact that $P_{X} = P_{-X}$ and the standard SDE representation of the RQF~\eqref{eq:brownian-1}, we conclude that the assumptions of Theorem \ref{th:main-sym} are satisfied. By Theorem \ref{th:rqf-one} we get the result.
\end{proof}

\begin{proof}[Proof of Theorem \ref{th:main-sym}]
    By definition we need to show that for arbitrary $u \in C^2(\bbS^{n-1}\times \bbS^{n-1})$ and any $\alpha \in [0,1]$ the measure $\rho_\alpha$ satisfies
    \[
    \int_{\bbS^{n-1} \times \bbS^{n-1}} (L_{XY} u) \rho_\alpha(\rmd x, \rmd y) = 0,
    \]
    where $L_{XY}$ is the generator of the two-point process.
    Using the structure of $\rho_\alpha$, we can be split the above integral into two parts:
    \[
    \begin{aligned}
    \MoveEqLeft \int_{\bbS^{n-1} \times \bbS^{n-1}}(L_{XY} u) \rho_\alpha(\rmd x, \rmd y) \\
    &= \alpha\int_{\bbS^{n-1}}(L_{XY} u)\big|_{y = x}\hat\rho(\rmd x) +(1-\alpha)\int_{\bbS^{n-1}}(L_{XY} u)\big|_{y = -x}\hat\rho(\rmd x) \\
    &=: \alpha I + (1-\alpha)II.
     \end{aligned}
    \]
    First note that $I = 0$ is trivially satisfied by an arbitrary pair of coupled diffusion processes since, by definition:
    \[
    (L_{XY}f)(x, y) \Big|_{y= x}= \lim_{t\to 0}\frac{1}{t} \bbE \left[f(X_t, Y_t) - f(X_0, Y_0)\big| X_0 = Y_0 = x\right] = (L_{X}f)(x, x),
    \]
    and thus for any invariant measure $\hat \rho$ of the process $X_t$ we get
    \[
    \int (L_{XY}f)(x, y) \Big|_{y= x}\hat\rho{\rmd x} = \int (L_{X}f)(x, x)\hat\rho{\rmd x} = 0.
    \]
    For the term $II$, using Lemma \ref{lem:generator-xy} below we analogously obtain 
   \[
   \begin{aligned}
   II &= \int \left( (\nabla_x u(x, y))^T G(x) + (\nabla_y u(x, y))^T G(y) \right)\big|_{y = -x} \hat\rho(\rmd x)  \\
   &\qquad + \int \left(  \tr\left(\nabla_x^2 u(x, y) D^X\right) + \tr\left(\nabla_y^2 u(x, y) D^Y\right)\right)\big|_{y = -x} \hat\rho(\rmd x) \\
   &\qquad +  2\int \tr\left(\nabla_x\nabla_y u(x, y) D^{XY}\right)\big|_{y = -x} \hat\rho(\rmd x)\\
   &=\int \left( \nabla_x u(x, -x)^T G(x)  + \tr\left(\nabla_{x}^2 u(x, -x) D^X\right)\right)\big|_{y = -x} \hat\rho(\rmd x) \\
   &= 2\int (L_X u(x, -x))\hat\rho(\rmd x) = 0,
   \end{aligned}
   \]
   where we used that $D(-x) =D(x)$, implying that
   \[
   G(y)\Big|_{y = -x} = F_0(-x) + \frac{1}{2}\nabla D(x) = -G(x)
   \]
   and thus
   \[
   \begin{aligned}
   \nabla_x u(x, -x)^T G(x)&= \left(\nabla_x u(x, y) - \nabla_y u(x, y) \right)^T G(x)\big|_{y = -x} \\
   &= \nabla_x u(x, y)^T G(x)\big|_{y = -x} + \nabla_y u(x, y) G(y)\big|_{y = -x}. 
   \end{aligned}
   \]
   Analogously, using that $D^{XY}\big|_{x=-y} = - D^X$ and $D^Y\big|_{x=-y} = D^X$, we obtain
   \[
   \begin{aligned}
   \nabla_{x}^2 u(x, -x) D^X\big|_{y = -x} &= \left(\nabla_{x}^2 u(x, y)  + \nabla_{y}^2 u(x, y) - 2\nabla_x\nabla_{y} u(x, y)\right) D^X\big|_{y = -x}\\
   &=\nabla_{x}^2 u(x, y) D^X\big|_{y = -x} + \nabla_{y}^2 u(x, y)D^Y\big|_{y = -x} + 2\nabla_x\nabla_{y} u(x, y) D^{XY}\big|_{y = -x},
   \end{aligned}
   \]
   and, hence, the result.
\end{proof}

\begin{lemma}[Generator of the two-point process]
\label{lem:generator-xy}
    Let $(X_t, Y_t)$ be the two-point process of an SDE as in Theorem \ref{th:main-sym}, then for any $u\in C^2(\bbS^{n-1}\times \bbS^{n-1})$ the following holds:
    \[
    (L_{XY}u)(x, y) %
    = (L_{X}u)(x, y) + (L_{Y}u)(x, y) + (L^{XY}u)(x, y),
    \]
    where the cross-term $L^{XY}u$ is given as
    $$ L^{XY}u = 2\tr\left(\nabla_x\nabla_y u D^{XY}\right), \quad D^{XY}(x, y) = \frac12\sum_{i} F_i(x) F_i(y)^T.$$

\end{lemma}
\begin{proof}
    Rewriting \eqref{eq:sym} in the Ito form we obtain
    \begin{align*}
    \rmd X_t &=F_0(X_t)\rmd t - \frac{1}{2}\nabla_x D(X_t) \rmd t + \sum F_i(X_t)\rmd B_t^i, \\
    D(X_t) &= \frac{1}{2} \sum_{i} F_i(X_t) F_i(X_t)^T,
    \end{align*}
    where $\nabla$ is the Euclidean gradient. Then a standard calculation for the coupled process yields
    \begin{align*}
        (L_{XY}u)(x, y) &= (\nabla_x u)^T G(x) + (\nabla_y u)^T G(y) + \tr\left(\nabla_x^2 u(x, y) D^X\right) + \tr\left(\nabla_y^2 u(x, y) D^Y\right) \notag \\
        &\qquad +\tr\left(\nabla_x\nabla_y u(x, y) D^{XY}\right) + \tr\left(\nabla_y\nabla_x u(x, y) D^{YX}\right)\notag \\
        &=(L_Xu)(x, y) + (L_Yu)(x, y) + 2\tr\left(\nabla_x\nabla_y u(x, y) D^{XY}\right), 
    \end{align*}
    where $G(x) = F_0(x)- \frac{1}{2}\nabla D(x)$ is the drift term and the covariance matrices satisfy 
    \[
    D^{XY}(x, y) = \frac12 \sum_{i} F_i(x) F_i(y)^T = \left(D^{YX}(x, y)\right)^T, 
    \]
    and $D^X(x, y) = D(x), \ D^Y(x, y) = D(y)$. 
\end{proof}

\subsection{Random attractors}
\label{sec:main-pw}

By Proposition \ref{prop:SDE}, the RQF~\eqref{eq:rqf}, equivalently written in terms of Eq.~\eqref{eq:brownian-1}, induces a smooth random dynamical system $(\theta, \varphi)$ on $\bbS^{n-1}$. By Theorem~\ref{theo:Correspondence}, the unique stationary measure $\bar \rho$ from Theorem~\ref{th:rqf-one} induces a unique invariant Markov measure 
$$ \mu(\dd \omega, \dd x) = \mathbb P(\dd \omega) \mu_\omega(\dd x),$$
where $\bbP$ is the Wiener measure on the canonical path space $\Omega$.
We now turn to characterizing its weak random point attractor (see Section~\ref{sec:attractors}) whose existence follows directly from Proposition \ref{prop:ExWeakPntAttr} and the compactness of $\bbS^{n-1}$.
\begin{theorem}[Random attractor consists of two poles]
\label{th:rqf-attractor}
    The invariant sample measure $\mu_{\omega}$ of the RQF is equidistributed on $N =2$ random points
    \[
    \mu_{\omega} = \frac{1}{2}\delta_{a(\omega)} +\frac{1}{2}\delta_{-a(\omega)}, 
    \]
    where $a(\omega): \Omega \to \bbS^{n-1}$ is an $\calF_{-\infty}^0$-measurable map. In particular, for any $x, y \in \bbS^{n-1}$
    \[
    \dist (\varphi(t, \omega, y), \varphi(t, \omega, x)) \to 0 \quad \text{or} \quad \dist (\varphi(t, \omega, y), -\varphi(t, \omega, x)) \to 0,
    \]
  $\Omega$-almost surely.
\end{theorem}
The proof of Theorem \ref{th:rqf-attractor} exploits the properties of an auxiliary process $Z_t = \left<X_t, Y_t\right>$ tracking the angle between two coupled processes $X_t$ and $Y_t$. Thus, before we proceed to the proof of Theorem \ref{th:rqf-attractor} we derive the explicit form of the dynamics of $Z_t$ in the following lemma.
\begin{lemma}[Process $Z_t$]
\label{lem:zt}
    Let $(X_t, Y_t)$ be the two-point process of the RQF as in Eq. \eqref{eq:rqf-two} and let $Z_t = \left<X_t, Y_t\right>$. Then the process $Z_t$ with values in $[-1,1]$ solves the It\^{o} SDE
    \[
    \rmd Z_t = -2Z_t(1-Z_t^2)\rmd t + Z_t \left(X_t^T\rmd Q_t X_t + Y_t^T\rmd Q_t Y_t\right) - 2X_t^T\rmd Q_t Y_t.
    \]
\end{lemma}
\begin{proof}
    We first transform Eq.~\eqref{eq:rqf} into It\^{o} form and then apply It\^{o}'s lemma.
    
    \emph{Step 1: It\^{o} form of the RQF.} Using the notation $\hat V^i_{j} = \hat V^i_{j}(X) =  (\delta_{i, j} - X^i X^j)$ and rewriting the RQF equation into tensor form we obtain
\[
\rmd X^i = -\frac{1}{\sqrt{2}}  \sum_{j, k} (\hat V^i_j X^k + \hat V^i_k X^j) \partial B^{j, k}
:= -\sum_{j, k} \sigma^i_{j, k}(X)  \partial B^{j, k}.
\]
Applying the Stratonovich to It\^{o} conversion, we therefore get
\[
\begin{aligned}
    \rmd X^i =- \frac{1}{\sqrt{2}}  \sum_{j, k} (\hat V^i_j X^k + \hat V^i_k X^j) \rmd B^{j, k} + \frac{1}{2} \sum_{j, k, l} \frac{\partial \sigma^i_{j, k}(x)}{\partial x_l}\sigma^l_{j, k}(x)\rmd t,
\end{aligned}
\]
where the correction term equals
\[
\begin{aligned}
\MoveEqLeft
    \sum_{j, k, l} \frac{\partial \sigma^i_{j, k}(x)}{\partial x_l}\sigma^l_{j, k}(x) = \frac12 \sum_{j, k, l} \left(\frac{\partial \hat V^i_j}{\partial x^l}x^k + \hat V^i_j \delta_{k, l} + \frac{\partial \hat V^i_k}{\partial x^l}x^j + \hat V^i_k\delta_{j, l}\right)(\hat V^l_j x^k + \hat V^l_k x^j) \\
    &=\frac12\sum_{j, k, l} \left(\frac{\partial \hat V^i_j}{\partial x^l}x^k + \frac{\partial \hat V^i_k}{\partial x^l}x^j \right)(\hat V^l_j x^k +\hat  V^l_k x^j) + \frac12\sum_{j, k, l} \left(\hat V^i_j \delta_{k, l} + \hat V^i_k\delta_{j, l}\right)(\hat V^l_j x^k + \hat V^l_k x^j)\\
    &:= I + II.
\end{aligned}
\]
For $II$, using $\sum_b V^a_b x^b = \sum_b V^b_a x^b = 0$, we obtain
\[
II = \frac12 \left( \sum_{j, l} \hat V^i_j \hat V^l_j x^l + \sum_{j, l, k} \delta_{j, l} \hat V^l_j \hat V^i_k x^k  +\sum_{k, l, j} \delta_{k, l} \hat V^l_k \hat V^i_j  x^j + \sum_{k, l}\hat V^i_k \hat V^l_k x^l \right) = 0.
\]
Differentiating the elements of the projection matrix,
\[
\frac{\partial \hat V^a_b}{\partial x^c} = -x^a\delta_{b, c} - x^b\delta_{a, c},
\]
we calculate for $I$, using again $\sum_b V^a_b x^b = \sum_b V^b_a x^b = 0$, 
\[
\begin{aligned}
I &= -\frac12\sum_{j, k, l} \left(x^kx^i\delta_{l, j} + x^kx^j\delta_{i, l} + x^jx^i\delta_{k, l} + x^jx^k\delta_{i, l} \right)(\hat V^l_j x^k +\hat  V^l_k x^j)\\
&=-\frac12\sum_{j, k, l} \Big(\hat V^l_jx^i(x^k)^2\delta_{l, j} + V^l_jx^j(x^k)^2\delta_{i, l} + V^l_jx^jx^ix^k\delta_{k, l} + V^l_jx^j(x^k)^2\delta_{i, l} \\
&\quad + \hat  V^l_k x^kx^ix^j\delta_{l, j} + \hat  V^l_kx^k(x^j)^2\delta_{i, l} + \hat  V^l_kx^i(x^j)^2\delta_{k, l} + \hat  V^l_kx^k(x^j)^2\delta_{i, l}\Big)\\
&= -\frac12\sum_{j, l, k}\left(\hat V^l_jx^i(x^k)^2\delta_{l, j} + \hat  V^l_kx^i(x^j)^2\delta_{k, l}\right) = -\sum_{l, k}\hat V^l_lx^i(x^k)^2 \\
&= -x^i\sum_{l, k}(1 - (x^l)^2)(x^k)^2 = -(n - 1)x^i.
\end{aligned}
\]
Combining the expressions we obtain the It\^{o} SDE
\[
\rmd X_t = -\frac{n-1}{2}X_t\rmd t - P_{X_t}\rmd Q_tX_t.
\]

 \emph{Step 2: Applying It\^{o}'s lemma to $f(x, y) := \left<x, y\right>$, we obtain} 
 \begin{align}
 \rmd Z_t &= \rmd f(X_t, Y_t)= X^T_t \rmd Y_t + Y^T_t\rmd X_t + \rmd [X_t, Y_t]  \notag \\
 &= -(n-1)Z_t \rmd t + \rmd [X_t, Y_t] + Z_t \left(X_t^T\rmd Q_t X_t + Y_t^T\rmd Q_t Y_t\right) - 2X_t^T\rmd Q_t Y_t, \label{eq:lem49-step2}
\end{align}
 where the quadratic covariation is $[X_t, Y_t] = \int_0^t q(Z_t) \rmd t$ and the function $q$ takes the form
 \[
 \begin{aligned}
 q(z) &= \sum_{i, j, k} \sigma^i_{j, k}(x)\sigma^i_{j, k}(y) = \frac{1}{2}\sum_{i, j, k} (\hat V^i_{j}(x)x^k + \hat V^i_{k}(x)x^j)(\hat V^i_{j}(y)y^k + \hat V^i_{k}(y)y^j) \\
 &=z\sum_{i, j}\hat V^i_{j}(x)\hat V^i_{j}(y) + \sum_{i, j, k}\hat V^i_{j}(x)\hat V^i_{k}(y)x^ky^j := A + B.
  \end{aligned}
 \]
 Calculating $A$, we obtain
 \[
 A = z\sum_{i, j}\hat V^i_{j}(x)\hat V^i_{j}(y) = z\sum_{i, j}(\delta_{i, j} - x^ix^j)(\delta_{i, j} - y^iy^j) = z(n - 2 + z^2),
 \]
 and analogously for $B$:
 \[
 \begin{aligned}
 B &= \sum_{i, j, k}(\delta_{i, j} -x^ix^j)(\delta_{i, k} - y^iy^k)x^ky^j \\
 &= \sum_{i}x^iy^i - \sum_{i, j} (x^i)^2x^jy^j - \sum_{i, k} (y^i)^2y^k x^k + \sum_{i, j, k}x^ix^jy^iy^kx^ky^j = z - 2z + z^3.
  \end{aligned}
 \]
 Combining $A$ and $B$, we conclude that $q(z) = nz - 3z +2z^3$ and, hence, plugging $[X_t, Y_t]$ into \eqref{eq:lem49-step2} we get the result.
\end{proof}
We are now ready to prove the main theorem.
\begin{proof}[Proof of Theorem \ref{th:rqf-attractor}]
For any two points $x, y \in \bbS^{n-1}$, the condition $x = y \lor x = -y$ is equivalent to $\left<x, y\right>^2 = 1$. Thus, to prove convergence to a polar/anti-polar configuration, it is sufficient to show that either $\left<X_t, Y_t\right> \to  1$ or $\left<X_t, Y_t\right> \to  -1$ . The structure of the proof is the following: we first define a $1$-dimensional diffusion $\hat Z_t \in [-1,1]$ distributionally equivalent to the process $Z_t = \left<X_t, Y_t\right>$. Then we show that $\hat Z_t \to \pm 1$ almost surely as $t\to \infty$. 
Finally, we prove that the weak random point attractor consists of two random points and characterize the sample measures.

    \emph{Step 1: tracking the scalar product.} Consider the coupled process \eqref{eq:rqf-two} and define the process $Z_t := \left<X_t, Y_t\right>$ for $t \in [0, \infty)$. Using Lemma \ref{lem:zt} we conclude that for initial conditions $X_0 = x, Y_0 = y, Z_0 = \left<x, y\right>$, the triplet $(X_t, Y_t, Z_t)$ follows the dynamics 
    \[
    \begin{aligned}
        \rmd X_t &= -P_{X_t} \partial Q_t X_t, \\
        \rmd Y_t &= -P_{Y_t} \partial Q_t Y_t, \\
        \rmd Z_t &= -2Z_t(1-Z_t^2)\rmd t + Z_t \left(X_t^T\rmd Q_t X_t + Y_t^T\rmd Q_t Y_t\right) - 2X_t^T\rmd Q_t Y_t.
    \end{aligned}
    \]
    Calculating the covariance $\Sigma$ of the process $Z_t$ we obtain
    \[
    \begin{aligned}
    \Sigma(z, x, y) &= \Sigma(z) = 2\sum_{i, j}(z(x_ix_j + y_iy_j) - (x_iy_j + y_ix_j))^2  \\
    &=2\left(z^2 \sum_{i, j} (x_ix_j + y_iy_j)^2 + \sum_{i, j} (x_iy_j + y_ix_j)^2 - 2z\sum_{i, j} (x_ix_j + y_iy_j)(x_iy_j + y_ix_j)\right) \\
    &= 2\left(z^2(2 + 2z^2) + (2+ 2z^2) - 8z^2\right) = 4(1 - z^2)^2,
    \end{aligned}
    \]
    where we used that $x, y \in \bbS^{n-1}$. Since $\Sigma $ is independent of $X$ and $Y$, the transition probability $p(z, z_0| t) = p(z)$ of the process $Z_t$ takes a closed form and is given by the solution of the Fokker-Planck equation
\[
\begin{aligned}
\partial_t p = \partial_z (2z(1-z^2) p(z)) + \partial_{zz}(2(1-z^2)^2p(z))
\end{aligned}
\]
with $p_0 = \delta_{z_0}$. The above Fokker-Planck equation corresponds to the process $\hat Z$ defined on the interval~$[-1,1]$ by
\[
\rmd \hat Z_t = -2\hat Z_t(1-\hat Z_t^2) \rmd t+ 2(1 - \hat Z_t^2) \rmd B_t,
\]
where $B_t$ is an arbitrary one-dimensional Brownian motion, implying that the distributional properties of the process $Z_t$ are identical to the ones of $\hat Z_t$.

 \emph{Step 2: $Z_t^2 \to 1$.} 
We use the \emph{Feller approach} to characterize the boundary behavior of the process $\hat Z_t$, see e.g. \cite[Section 5.5]{karatzas2012brownian}. Fix $c\in(-1, 1)$ and recall that the scale function $s(z)$ has the following expression:
\[
\begin{aligned}
        s(z) &= \int_c^{z} \exp\left(\int_c^y \frac{x}{(1-x^2) }\rmd x \right)\rmd y = \int_c^{z} \exp\left( -\frac12(\log(1-y^2) - \log(1-c^2))\right)\rmd y \\
        &=\sqrt{(1-c^2)}\int^z_c\frac{1}{\sqrt{(1-y^2)}}\rmd y =\sqrt{(1-c^2)}(\arcsin (z) - \arcsin (c)).
        \end{aligned}
    \]
    Since $s(-1) > -\infty$ and $s(1)< +\infty$, both boundaries are attractive and, thus, by \cite[Proposition 5.22] {karatzas2012brownian}, we obtain
    \[
    \bbP(\hat Z_t \to 1) + \bbP(\hat Z_t \to -1) =1.
    \]
Since the processes $Z_t$ and $\hat Z_t$ have the same law, we conclude that for any initial conditions $X_0, Y_0 \in \bbS^{n-1}$, the processes $X_t$ and $Y_t$ satisfy: 
\[
\bbP(\left<X_t, Y_t\right> \to 1) + \bbP(\left<X_t, Y_t\right> \to -1) =1.
\]

    \emph{Step 3: from distributions to attractors.}
    From Step 2 we know that the coupled processes $(X_t, Y_t)$ almost surely converge to either a polar or an anti-polar configuration. In particular, we immediately obtain that the induced cocycle statisfies
    \[
    \dist (\varphi(t, \omega, y), \varphi(t, \omega, x)) \to 0 \quad \text{or} \quad \dist (\varphi(t, \omega, y), -\varphi(t, \omega, x)) \to 0,
    \]
    almost surely, for any fixed $x,y \in \bbS^{n-1}$.
    It is immediate that the weak random point attractor $A(\omega) = \operatorname{supp}(\mu_\omega)$ consists of at most two points $\{a(\omega), -a(\omega)\}$, where $a(\omega)$ is $\calF_{-\infty}^0$-measurable. 
    But note that, by a symmetry argument, the equality $\varphi(t, \omega, x) = \xi$ always implies that
$\varphi(t, \omega, -x) = -\xi$
    and, thus, $A(\omega)$ consists of exactly two points. 
    Finally, since $A(\omega)$ is discrete, by Proposition \ref{prop:discrete_cont}, the sample measures $\mu_\omega$ are equidistributed on $A(\omega)$. 
\end{proof}
\begin{remark}
    Theorem \ref{th:rqf-attractor} (and Step 2 of the proof, in particular) can be interpreted as an extension of the approach to prove (partial) synchronization, as introduced in \cite[Lemma 2.8.]{cranston2016weak}, to the case when the random attractor is discrete but not a singleton. Instead of tracking the distance between points as in \cite{cranston2016weak}, we track the scalar product between them. We also remark that a similar argument has been recently used in \cite{fedorov2026clustering} in the context of transformer models. An alternative approach to prove partial synchronization of the RQF may make use of the results by Raimond \cite{raimond1999flots}.
\end{remark}

\begin{proposition}[Convergence to particular invariant measure]
Let $P_t$ be the semigroup of the two-point process \eqref{eq:rqf-two} and $P^*_t$ be its adjoint, then 
\[
(P_t)^*(\bar \rho \times \bar \rho) \to
\bar\rho(\rmd x)\times\left(\frac{1}{2}\delta_x(\rmd y) + \frac{1}{2}\delta_{-x}(\rmd y)\right).
\]
\end{proposition}
\begin{proof}
    Since $\bar \rho$ is the unique ergodic measure of the RQF, the result follows directly from \cite[Proposition 2.6]{baxendale1991statistical}. 
\end{proof}
\begin{remark}[Solutions of random Wasserstein gradient flow]
    Consider the random Wasserstein gradient flow discussed in Remark \ref{rem:random-wgf} and define its solution $\rho_t : [0, \infty)\times \Omega \to \calP(\bbS^{n-1})$ as the pushforward of $\rho_0$ under the map $\varphi(t, \omega, \cdot)$
    \[
    \rho_t = \varphi(t, \omega, \cdot)^*\rho_0.
    \]
    Then $\rho_t $ necessarily satisfies
    \[
    \rho_t - \varphi(t, \omega, \cdot)^*\rho_\omega \stackrel{w}{\to} 0, 
    \]
    for some probability measure $\rho_\omega$ satisfying $\supp \rho_\omega \subseteq A(\omega)$, $\Omega$-a.s.
    In other words, in the long-time limit, $\rho_t$ becomes effectively supported solely on the image of the random attractor. 

    Note that this property is reminiscent of the behaviour of deterministic Wasserstein gradient flows of a potential energy $F$, where one usually expects the limiting measure to be supported on the set of critical points of $F$:
    \[
    \rho_t - \rho_{\infty} \stackrel{w}{\to} 0, \quad  \supp \rho_{\infty} \subseteq \Crit_F,
    \]
    in the long-time limit. We also remark that the saddle points of the deterministic quadratic form do not survive the randomization.
\end{remark}

\section{Discussion and Future Work}
\label{sec:discussion}

We conclude with discussing some additional aspects of the RQF approach. We start with an alternative proof of Theorem~\ref{th:rqf-attractor} for the RQF on $\bbS^1$.

\subsection{Relation to harmonic noise on the circle}
Remarkably, the RQF model on the circle $\bbS^1$ coincides with a common example of an RDS which has negative maximal Lyapunov exponent but does not exhibit convergence to a singleton. 
In more detail, let $X_t = (X_t^1, X_t^2) \in \bbS^1 \subset \bbR^2$ follow the dynamics \eqref{eq:rqf}. Expanding the projection $P_{X_t}$ yields
\[
\begin{pmatrix}
\rmd X^1_t\\
\rmd X^2_t
\end{pmatrix} = \begin{pmatrix}
 1-(X^1_t)^2 & -X^1_t X^2_t\\
 -X^2_tX^1_t & 1- (X^2_t)^2
\end{pmatrix}  \begin{pmatrix}
\partial B_{11} &  \frac{1}{2}\partial(B_{12} + B_{21})\\
  \frac{1}{2}\partial(B_{12} + B_{21}) & \partial B_{22} 
\end{pmatrix} \begin{pmatrix}
 X^1_t\\
 X^2_t
\end{pmatrix}.
\]
Since $X_t \in \bbS^{1}$, we have $(X^2_t)^2 \equiv 1 - (
X^1_t)^2$. Therefore it is sufficient to consider the dynamics of the first coordinate $X^1_t$, which satisfies
\begin{align*}
\rmd X^1_t &= \begin{pmatrix}
 1-(X^1_t)^2 & -X^1_t \sqrt{1- (X^1_t)^2}
\end{pmatrix} 
\begin{pmatrix}
 \partial B_{11} &  \frac{1}{2}\partial(B_{12} + B_{21})\\
  \frac{1}{2}\partial(B_{12} + B_{21}) & \partial B_{22}
\end{pmatrix}
\begin{pmatrix}
 X^1_t\\
 \sqrt{1 - (X^1_t)^2} \end{pmatrix}
 \\
 &= (1-(X^1_t)^2)\left(X^1_t\partial B_{11} + \frac{1}{2}\sqrt{1 - (X^1_t)^2}\partial (B_{12}+B_{21})\right) \\
 &\qquad - X^1_t \sqrt{1- (X^1_t)^2}\left(\frac12 X^1_t\partial (B_{12}+B_{21}) + \sqrt{1 - (X^1_t)^2}\partial B_{22}\right) \\
 &= X^1_t(1-(X^1_t)^2) \partial (B_{11} - B_{22} )+ 
 \frac{1}{2}\sqrt{1 - (X^1_t)^2} \left(1-2(X^1_t)^2 \right) \partial (B_{12}+B_{21}).
\end{align*}
After the change of variables $X^1_t = \cos \phi_t$ we obtain the system
\begin{align}
\rmd \phi_t &= -(1-(X^1_t)^2)^{-1/2}\rmd X^1_t \notag\\
&= - \cos \phi_t \sin \phi_t   \partial\left(B_{11} - B_{22}\right) + \frac12(\cos\phi_t^2 - \sin \phi_t^2)  \partial (B_{12}+B_{21}) \notag \\
&= \frac{1}{2}\sin 2\phi_t  \partial\left( B_{22} - B_{11}\right) + \frac{1}{2} \cos2\phi_t \partial (B_{12}+B_{21}). \label{eq:cos2phi}
\end{align}
Variants of the bi-harmonic model \eqref{eq:cos2phi} have already been discussed in the context of synchronization, see e.g.~\cite{baxendale1991statistical} or \cite{FGS1}. In particular, as follows from the calculation in \cite{baxendale2006asymptotic}, the maximal Lyapunov exponent of the system is negative and thus the random attractor is discrete by Proposition \ref{prop:discrete}. For the convenience of the reader we provide a short proof of this result below.
\begin{proposition}
    Let $\phi_t$ be the solution process of  Eq.~\eqref{eq:cos2phi} on $\bbR / 2 \pi \bbZ \cong \bbS^1$. Then the 
    \begin{itemize}
        \item the maximal Lyapunov exponent of the induced RDS is $\Lambda = -1$,
        \item the sample measures of the associated invariant Markov measure are supported on two polar points 
        \[
        \mu_\omega = \frac{1}{2}\delta_{a(\omega)} + \frac{1}{2}\delta_{a(\omega) + \pi},
        \]
        where $a$ is an $\mathcal F_{- \infty}^0$-measurable random variable.
    \end{itemize}
\end{proposition}
\begin{proof}
    \emph{Step 1: calculating $\Lambda$.} Let $\psi_t = 2\phi_t$. Then $\psi_t$ solves the equation
    \[
    \rmd \psi_t = \sin \psi_t  \partial\left( B_{22} - B_{11}\right) + \cos\psi_t \partial (B_{12}+B_{21}),
    \]
    defined on the interval $[0, 4\pi)$. Note that both driving processes $(B_{22} - B_{11})$ and $(B_{12}+B_{21})$ are rescaled independent Wiener processes and thus we can rewrite the system as
    \begin{equation}
    \label{eq:lyap-bax}
    \rmd \psi_t =\sqrt{2}\sin \psi_t  \partial W_1 + \sqrt{2} \cos\psi_t \partial W_2,
    \end{equation}
    with independent Brownian motions $W_1, W_2$. It follows from \cite[Example 2]{baxendale2006asymptotic} that the Lyapunov exponent of \eqref{eq:lyap-bax} is $\Lambda_\psi = -1$. Changing the variables back to $\phi$, we thus conclude that the Lyapunov exponent of the bi-harmonic model \eqref{eq:cos2phi} is $\Lambda = \Lambda_\psi = -1$.

    \emph{Step 2: employing the symmetry.} From Proposition \ref{prop:discrete}, we can conclude that $\mu_\omega$ is discrete. Moreover, using the symmetry argument as in the proof of Theorem \ref{th:rqf-attractor}, we deduce that the random attractor consists of at least two points. We will now show by contradiction that the random attractor almost surely consists of exactly two points. 
    
    Note that the factorization of $\psi_t$, namely the dynamical system 
    \[
    \tilde \psi_t = \psi_t\mod 2\pi,
    \]
    is identical to the dynamical system 
    \[
    \rmd \hat \psi_t =\sqrt{2}\sin \hat \psi_t  \partial W_1 + \sqrt{2} \cos\hat \psi_t \partial W_2,
    \]
    defined on $[0, 2\pi)$. From \cite[Example 2]{baxendale2006asymptotic}, it follows that the invariant sample measure $\hat \mu_\omega$ associated with the stochastic flow of $\hat \psi_t$ is almost surely concentrated on a single point $\hat a(\omega)$. 
    Let now $A_{\psi}(\omega) = \supp(\tilde \mu_\omega)$ be the random point attractor of the RDS $(\theta, \varphi_{\psi})$ induced by the SDE \eqref{eq:lyap-bax} on $[0, 4 \pi)$. 
Assume that the random attractor $A_{\psi}(\omega)$ %
consists of $N> 2$ distinct points, then the corresponding random attractor of the factorized RDS $A_{\hat \psi}(\omega) = \supp(\hat \mu_\omega)$ consists of $\hat N\geq \lfloor \frac{N+1}{2}\rfloor$ points, which leads to a contradiction. We, therefore, conclude that $A_\psi(\omega) = \{\tilde a_1(\omega), \tilde a_2(\omega)\}$, where
    \[
    \tilde a_2(\omega) = \tilde a_1(\omega) + 2\pi, \quad \tilde a_1(\omega) \in [0, 2\pi).
    \]
    Finally, by definition of the random point attractor for the RDS with cocycle $\varphi_{\psi}$, we obtain that, for all $x\in [0, 4\pi)$,
\[
\dist\big(\varphi_{\psi}(t, \theta_{-t}\omega, x), a(\omega)\big) \to 0, \quad \text{ for some } \ a(\omega) \in A_\psi(\omega)=\{\tilde a_1(\omega), \tilde a_2(\omega)\},
\]
at least in probability.
Since $\psi_t = 2 \phi_t$, this implies that, in probability,
\[
\dist\left(\varphi_{\phi}(t, \theta_{-t}\omega, \frac{x}{2}), \frac{a(\omega)}{2}\right) = \frac12\dist\big(\varphi_{\psi}(t, \theta_{-t}\omega, x), a(\omega)\big)\to 0.
\]
Hence, we have that
$$A_{\phi}(\omega) = \left\{\frac{\tilde a_1(\omega)}{2}, \frac{\tilde a_2(\omega)}{2} \right\}$$
is the minimal weak random point attractor of the RDS induced by \eqref{eq:cos2phi}, yielding the result, again by Proposition \ref{prop:discrete}.
\end{proof}
\subsection{Including the bias term}
\label{sec:bias}
So far, we have considered the dynamics of the purely quadratic random functional. At the same time, the motivating example of linear layers in transformers also includes a bias term. Including such a (random) bias results in the model
\begin{equation}
\label{eq:bias}
\rmd X_t = -\nabla \left(\frac12 X_t^T \partial Q_t X_t + X^T_t \partial W_t\right),
\end{equation}
where $W_t$ is an $n$-dimensional Wiener process, independent from the process $Q_t$. Extending the approach from the introduction, we observe that the bias term is a formal gradient flow of a linear functional
\[
\nabla ( X^T_t \partial W_t) = \nabla F^{lin}_{\partial W_t} = \nabla \left<X, \partial W_t\right>.
\]
The deterministic counterpart of $F^{lin}$ is $F_v = \left<x, v\right>$, for some $v\in\bbR^n$; hence, for arbitrary $v\neq 0$, the functional $F_v$ has a unique minimizer on the unit sphere $x^* = -v/\|v\|$.

Reducing the model to the pure random bias results in the classical formulation of the Brownian motion on the sphere, namely the dynamics of the form
\begin{equation}
\label{eq:pure_bias}
    \rmd X_t = -\nabla (X^T_t \partial W_t) = -P_{X_t}\partial W_t.
\end{equation}
The two-point motion of this Brownian motion on the sphere has been studied by P. Baxendale in \cite{baxendale1986asymptotic}, where it was shown that the random attractor of the induced RDS is, in fact, a singleton. Thus, the Brownian motion model~\eqref{eq:pure_bias} provides another example where the structure of the attractor for a random gradient flow and the minimizer of the deterministic counterpart are intimately related.

Since the RQF is a Wiener process, the combined model \eqref{eq:bias} is, in law, also a (rescaled) Brownian motion. At the same time, the nature of the attractor of \eqref{eq:bias} is not obvious. 
In this work, we have seen that the quadratic energy functional leads to a bi-polar configuration while the linear one results in particles forming a single cluster; hence, we anticipate the existence of a phase transition, depending on the relative contribution of the two components. We also remark that a similar effect has been described in  \cite{alvarez2026perceptrons}, where, however, the single cluster regime was enforced by a deterministic drift.

\subsection{More general architectures}
In this work, we have restricted the transformer model to a linear activation function in the Feed-Forward layers, discarding the effects coming from the non-linearity. Hence, another future extension may entail the explicit characterization of the effect of the activation function on the dynamics of the corresponding random flow. Note that, for the deterministic model, such an effect was recently considered in \cite{alvarez2026perceptrons}. 

Taking the formal quadratic expansion of the dynamics with non-linearity, we obtain
\[
\rmd X_t = \sigma (\partial A_t X_t) \approx \sigma(0) +\sigma'(0)\partial A_t X_t + \frac{1}{2}\sigma''(0)[\partial A_t X_t, \partial A_t X_t] + o([\partial A_t, \partial A_t]),
\]
i.e., the correction appearing due to a non-trivial activation function has a deterministic term of leading order. Thus, we expect that the methods of this work together with the ones from \cite{fedorov2026clustering} may be useful to study the long-term behavior of the corresponding dynamical system. Multiple other generalizations will arise from combining feed-forward and self-attention dynamics in the common noise framework and from varying the dynamics of the driving noise. 
\bibliographystyle{alpha}
\bibliography{biblio}
\end{document}